\documentclass[11pt, reqno]{amsart}
\usepackage{amssymb}
\usepackage{eucal}
\usepackage{amsmath}
\usepackage{amscd}
\usepackage[dvips]{color}
\usepackage{multicol}
\usepackage[all]{xy}           
\usepackage{graphicx}
\usepackage{color}
\usepackage{colordvi}
\usepackage{xspace}
\usepackage{tikz}

\usepackage{ifpdf}
\ifpdf
 \usepackage[colorlinks,final,hyperindex]{hyperref}
\else
 \usepackage[colorlinks,final,hyperindex,hypertex]{hyperref}
\fi

\begin{document}
\newcommand {\emptycomment}[1]{} 

\baselineskip=14pt
\newcommand{\nc}{\newcommand}
\newcommand{\delete}[1]{}
\nc{\mfootnote}[1]{\footnote{#1}} 
\nc{\todo}[1]{\tred{To do:} #1}

\delete{
\nc{\mlabel}[1]{\label{#1}}  
\nc{\mcite}[1]{\cite{#1}}  
\nc{\mref}[1]{\ref{#1}}  
\nc{\meqref}[1]{\ref{#1}} 
\nc{\mbibitem}[1]{\bibitem{#1}} 
}

\nc{\mlabel}[1]{\label{#1}  
{\hfill \hspace{1cm}{\bf{{\ }\hfill(#1)}}}}
\nc{\mcite}[1]{\cite{#1}{{\bf{{\ }(#1)}}}}  
\nc{\mref}[1]{\ref{#1}{{\bf{{\ }(#1)}}}}  
\nc{\meqref}[1]{\eqref{#1}{{\bf{{\ }(#1)}}}} 
\nc{\mbibitem}[1]{\bibitem[\bf #1]{#1}} 

\newtheorem{thm}{Theorem}[section]
\newtheorem{lem}[thm]{Lemma}
\newtheorem{cor}[thm]{Corollary}
\newtheorem{pro}[thm]{Proposition}
\theoremstyle{definition}
\newtheorem{defi}[thm]{Definition}
\newtheorem{ex}[thm]{Example}
\newtheorem{rmk}[thm]{Remark}
\newtheorem{pdef}[thm]{Proposition-Definition}
\newtheorem{condition}[thm]{Condition}

\renewcommand{\labelenumi}{{\rm(\alph{enumi})}}
\renewcommand{\theenumi}{\alph{enumi}}

\nc{\tred}[1]{\textcolor{red}{#1}}
\nc{\tblue}[1]{\textcolor{blue}{#1}}
\nc{\tgreen}[1]{\textcolor{green}{#1}}
\nc{\tpurple}[1]{\textcolor{purple}{#1}}
\nc{\btred}[1]{\textcolor{red}{\bf #1}}
\nc{\btblue}[1]{\textcolor{blue}{\bf #1}}
\nc{\btgreen}[1]{\textcolor{green}{\bf #1}}
\nc{\btpurple}[1]{\textcolor{purple}{\bf #1}}

\nc{\jt}[1]{\textcolor{yellow}{JT:#1}}
\nc{\cm}[1]{\textcolor{red}{CM:#1}}
\nc{\liu}[1]{\textcolor{blue}{Liu:#1}}
\nc{\jf}[1]{\textcolor{blue}{#1}}


\nc{\twovec}[2]{\left(\begin{array}{c} #1 \\ #2\end{array} \right )}
\nc{\threevec}[3]{\left(\begin{array}{c} #1 \\ #2 \\ #3 \end{array}\right )}
\nc{\twomatrix}[4]{\left(\begin{array}{cc} #1 & #2\\ #3 & #4 \end{array} \right)}
\nc{\threematrix}[9]{{\left(\begin{matrix} #1 & #2 & #3\\ #4 & #5 & #6 \\ #7 & #8 & #9 \end{matrix} \right)}}
\nc{\twodet}[4]{\left|\begin{array}{cc} #1 & #2\\ #3 & #4 \end{array} \right|}

\nc{\rk}{\mathrm{r}}
\newcommand{\g}{\mathfrak g}
\newcommand{\h}{\mathfrak h}
\newcommand{\pf}{\noindent{$Proof$.}\ }
\newcommand{\frkg}{\mathfrak g}
\newcommand{\frkh}{\mathfrak h}
\newcommand{\Id}{\rm{Id}}
\newcommand{\gl}{\mathfrak {gl}}
\newcommand{\ad}{\mathrm{ad}}
\newcommand{\add}{\frka\frkd}
\newcommand{\frkb}{\mathfrak b}
\newcommand{\frkc}{\mathfrak c}
\newcommand{\frkd}{\mathfrak d}
\newcommand{\frkl}{\mathfrak l}
\newcommand{\frkr}{\mathfrak r}
\newcommand{\sgn}{\mathrm{sgn}}
\newcommand{\dM}{\mathrm{d}}
\newcommand{\perm}{\mathbb S}
\newcommand {\comment}[1]{{\marginpar{*}\scriptsize\textbf{Comments:} #1}}

\nc{\tforall}{\text{ for all }}

\nc{\svec}[2]{{\tiny\left(\begin{matrix}#1\\
#2\end{matrix}\right)\,}}  
\nc{\ssvec}[2]{{\tiny\left(\begin{matrix}#1\\
#2\end{matrix}\right)\,}} 

\nc{\typeI}{local cocycle $3$-Lie bialgebra\xspace}
\nc{\typeIs}{local cocycle $3$-Lie bialgebras\xspace}
\nc{\typeII}{double construction $3$-Lie bialgebra\xspace}
\nc{\typeIIs}{double construction $3$-Lie bialgebras\xspace}

\nc{\bia}{{$\mathcal{P}$-bimodule ${\bf k}$-algebra}\xspace}
\nc{\bias}{{$\mathcal{P}$-bimodule ${\bf k}$-algebras}\xspace}

\nc{\rmi}{{\mathrm{I}}}
\nc{\rmii}{{\mathrm{II}}}
\nc{\rmiii}{{\mathrm{III}}}
\nc{\pr}{{\mathrm{pr}}}
\newcommand{\huaA}{\mathcal{A}}

\nc{\OT}{constant $\theta$-}
\nc{\T}{$\theta$-}
\nc{\IT}{inverse $\theta$-}

\nc{\pll}{\beta}
\nc{\plc}{\epsilon}

\nc{\ass}{{\mathit{Ass}}}
\nc{\lie}{{\mathit{Lie}}}
\nc{\comm}{{\mathit{Comm}}}
\nc{\dend}{{\mathit{Dend}}}
\nc{\zinb}{{\mathit{Zinb}}}
\nc{\tdend}{{\mathit{TDend}}}
\nc{\prelie}{{\mathit{preLie}}}
\nc{\postlie}{{\mathit{PostLie}}}
\nc{\quado}{{\mathit{Quad}}}
\nc{\octo}{{\mathit{Octo}}}
\nc{\ldend}{{\mathit{ldend}}}
\nc{\lquad}{{\mathit{LQuad}}}

 \nc{\adec}{\check{;}} \nc{\aop}{\alpha}
\nc{\dftimes}{\widetilde{\otimes}} \nc{\dfl}{\succ} \nc{\dfr}{\prec}
\nc{\dfc}{\circ} \nc{\dfb}{\bullet} \nc{\dft}{\star}
\nc{\dfcf}{{\mathbf k}} \nc{\apr}{\ast} \nc{\spr}{\cdot}
\nc{\twopr}{\circ} \nc{\tspr}{\star} \nc{\sempr}{\ast}
\nc{\disp}[1]{\displaystyle{#1}}
\nc{\bin}[2]{ (_{\stackrel{\scs{#1}}{\scs{#2}}})}  
\nc{\binc}[2]{ \left (\!\! \begin{array}{c} \scs{#1}\\
    \scs{#2} \end{array}\!\! \right )}  
\nc{\bincc}[2]{  \left ( {\scs{#1} \atop
    \vspace{-.5cm}\scs{#2}} \right )}  
\nc{\sarray}[2]{\begin{array}{c}#1 \vspace{.1cm}\\ \hline
    \vspace{-.35cm} \\ #2 \end{array}}
\nc{\bs}{\bar{S}} \nc{\dcup}{\stackrel{\bullet}{\cup}}
\nc{\dbigcup}{\stackrel{\bullet}{\bigcup}} \nc{\etree}{\big |}
\nc{\la}{\longrightarrow} \nc{\fe}{\'{e}} \nc{\rar}{\rightarrow}
\nc{\dar}{\downarrow} \nc{\dap}[1]{\downarrow
\rlap{$\scriptstyle{#1}$}} \nc{\uap}[1]{\uparrow
\rlap{$\scriptstyle{#1}$}} \nc{\defeq}{\stackrel{\rm def}{=}}
\nc{\dis}[1]{\displaystyle{#1}} \nc{\dotcup}{\,
\displaystyle{\bigcup^\bullet}\ } \nc{\sdotcup}{\tiny{
\displaystyle{\bigcup^\bullet}\ }} \nc{\hcm}{\ \hat{,}\ }
\nc{\hcirc}{\hat{\circ}} \nc{\hts}{\hat{\shpr}}
\nc{\lts}{\stackrel{\leftarrow}{\shpr}}
\nc{\rts}{\stackrel{\rightarrow}{\shpr}} \nc{\lleft}{[}
\nc{\lright}{]} \nc{\uni}[1]{\tilde{#1}} \nc{\wor}[1]{\check{#1}}
\nc{\free}[1]{\bar{#1}} \nc{\den}[1]{\check{#1}} \nc{\lrpa}{\wr}
\nc{\curlyl}{\left \{ \begin{array}{c} {} \\ {} \end{array}
    \right .  \!\!\!\!\!\!\!}
\nc{\curlyr}{ \!\!\!\!\!\!\!
    \left . \begin{array}{c} {} \\ {} \end{array}
    \right \} }
\nc{\leaf}{\ell}       
\nc{\longmid}{\left | \begin{array}{c} {} \\ {} \end{array}
    \right . \!\!\!\!\!\!\!}
\nc{\ot}{\otimes} \nc{\sot}{{\scriptstyle{\ot}}}
\nc{\otm}{\overline{\ot}}
\nc{\ora}[1]{\stackrel{#1}{\rar}}
\nc{\ola}[1]{\stackrel{#1}{\la}}
\nc{\pltree}{\calt^\pl}
\nc{\epltree}{\calt^{\pl,\NC}}
\nc{\rbpltree}{\calt^r}
\nc{\scs}[1]{\scriptstyle{#1}} \nc{\mrm}[1]{{\rm #1}}
\nc{\dirlim}{\displaystyle{\lim_{\longrightarrow}}\,}
\nc{\invlim}{\displaystyle{\lim_{\longleftarrow}}\,}
\nc{\mvp}{\vspace{0.5cm}} \nc{\svp}{\vspace{2cm}}
\nc{\vp}{\vspace{8cm}} \nc{\proofbegin}{\noindent{\bf Proof: }}
\nc{\proofend}{$\blacksquare$ \vspace{0.5cm}}
\nc{\freerbpl}{{F^{\mathrm RBPL}}}
\nc{\sha}{{\mbox{\cyr X}}}  
\nc{\ncsha}{{\mbox{\cyr X}^{\mathrm NC}}} \nc{\ncshao}{{\mbox{\cyr
X}^{\mathrm NC,\,0}}}
\nc{\shpr}{\diamond}    
\nc{\shprm}{\overline{\diamond}}    
\nc{\shpro}{\diamond^0}    
\nc{\shprr}{\diamond^r}     
\nc{\shpra}{\overline{\diamond}^r}
\nc{\shpru}{\check{\diamond}} \nc{\catpr}{\diamond_l}
\nc{\rcatpr}{\diamond_r} \nc{\lapr}{\diamond_a}
\nc{\sqcupm}{\ot}
\nc{\lepr}{\diamond_e} \nc{\vep}{\varepsilon} \nc{\labs}{\mid\!}
\nc{\rabs}{\!\mid} \nc{\hsha}{\widehat{\sha}}
\nc{\lsha}{\stackrel{\leftarrow}{\sha}}
\nc{\rsha}{\stackrel{\rightarrow}{\sha}} \nc{\lc}{\lfloor}
\nc{\rc}{\rfloor}
\nc{\tpr}{\sqcup}
\nc{\nctpr}{\vee}
\nc{\plpr}{\star}
\nc{\rbplpr}{\bar{\plpr}}
\nc{\sqmon}[1]{\langle #1\rangle}
\nc{\forest}{\calf}
\nc{\altx}{\Lambda_X} \nc{\vecT}{\vec{T}} \nc{\onetree}{\bullet}
\nc{\Ao}{\check{A}}
\nc{\seta}{\underline{\Ao}}
\nc{\deltaa}{\overline{\delta}}
\nc{\trho}{\tilde{\rho}}

\nc{\rpr}{\circ}
\nc{\dpr}{{\tiny\diamond}}
\nc{\rprpm}{{\rpr}}

\nc{\mmbox}[1]{\mbox{\ #1\ }} \nc{\ann}{\mrm{ann}}
\nc{\Aut}{\mrm{Aut}} \nc{\can}{\mrm{can}}
\nc{\twoalg}{{two-sided algebra}\xspace}
\nc{\colim}{\mrm{colim}}
\nc{\Cont}{\mrm{Cont}} \nc{\rchar}{\mrm{char}}
\nc{\cok}{\mrm{coker}} \nc{\dtf}{{R-{\rm tf}}} \nc{\dtor}{{R-{\rm
tor}}}
\renewcommand{\det}{\mrm{det}}
\nc{\depth}{{\mrm d}}
\nc{\Div}{{\mrm Div}} \nc{\End}{\mrm{End}} \nc{\Ext}{\mrm{Ext}}
\nc{\Fil}{\mrm{Fil}} \nc{\Frob}{\mrm{Frob}} \nc{\Gal}{\mrm{Gal}}
\nc{\GL}{\mrm{GL}} \nc{\Hom}{\mrm{Hom}} \nc{\hsr}{\mrm{H}}
\nc{\hpol}{\mrm{HP}} \nc{\id}{\mrm{id}} \nc{\im}{\mrm{im}}
\nc{\incl}{\mrm{incl}} \nc{\length}{\mrm{length}}
\nc{\LR}{\mrm{LR}} \nc{\mchar}{\rm char} \nc{\NC}{\mrm{NC}}
\nc{\mpart}{\mrm{part}} \nc{\pl}{\mrm{PL}}
\nc{\ql}{{\QQ_\ell}} \nc{\qp}{{\QQ_p}}
\nc{\rank}{\mrm{rank}} \nc{\rba}{\rm{RBA }} \nc{\rbas}{\rm{RBAs }}
\nc{\rbpl}{\mrm{RBPL}}
\nc{\rbw}{\rm{RBW }} \nc{\rbws}{\rm{RBWs }} \nc{\rcot}{\mrm{cot}}
\nc{\rest}{\rm{controlled}\xspace}
\nc{\rdef}{\mrm{def}} \nc{\rdiv}{{\rm div}} \nc{\rtf}{{\rm tf}}
\nc{\rtor}{{\rm tor}} \nc{\res}{\mrm{res}} \nc{\SL}{\mrm{SL}}
\nc{\Spec}{\mrm{Spec}} \nc{\tor}{\mrm{tor}} \nc{\Tr}{\mrm{Tr}}
\nc{\mtr}{\mrm{sk}}

\nc{\ab}{\mathbf{Ab}} \nc{\Alg}{\mathbf{Alg}}
\nc{\Algo}{\mathbf{Alg}^0} \nc{\Bax}{\mathbf{Bax}}
\nc{\Baxo}{\mathbf{Bax}^0} \nc{\RB}{\mathbf{RB}}
\nc{\RBo}{\mathbf{RB}^0} \nc{\BRB}{\mathbf{RB}}
\nc{\Dend}{\mathbf{DD}} \nc{\bfk}{{\bf k}} \nc{\bfone}{{\bf 1}}
\nc{\base}[1]{{a_{#1}}} \nc{\detail}{\marginpar{\bf More detail}
    \noindent{\bf Need more detail!}
    \svp}
\nc{\Diff}{\mathbf{Diff}} \nc{\gap}{\marginpar{\bf
Incomplete}\noindent{\bf Incomplete!!}
    \svp}
\nc{\FMod}{\mathbf{FMod}} \nc{\mset}{\mathbf{MSet}}
\nc{\rb}{\mathrm{RB}} \nc{\Int}{\mathbf{Int}}
\nc{\Mon}{\mathbf{Mon}}
\nc{\remarks}{\noindent{\bf Remarks: }}
\nc{\OS}{\mathbf{OS}} 
\nc{\Rep}{\mathbf{Rep}}
\nc{\Rings}{\mathbf{Rings}} \nc{\Sets}{\mathbf{Sets}}
\nc{\DT}{\mathbf{DT}}

\nc{\BA}{{\mathbb A}} \nc{\CC}{{\mathbb C}} \nc{\DD}{{\mathbb D}}
\nc{\EE}{{\mathbb E}} \nc{\FF}{{\mathbb F}} \nc{\GG}{{\mathbb G}}
\nc{\HH}{{\mathbb H}} \nc{\LL}{{\mathbb L}} \nc{\NN}{{\mathbb N}}
\nc{\QQ}{{\mathbb Q}} \nc{\RR}{{\mathbb R}} \nc{\BS}{{\mathbb{S}}} \nc{\TT}{{\mathbb T}}
\nc{\VV}{{\mathbb V}} \nc{\ZZ}{{\mathbb Z}}
\newcommand{\frka}{\mathfrak a}
\newcommand{\frks}{\mathfrak s}

\nc{\calao}{{\mathcal A}} \nc{\cala}{{\mathcal A}}
\nc{\calc}{{\mathcal C}} \nc{\cald}{{\mathcal D}}
\nc{\cale}{{\mathcal E}} \nc{\calf}{{\mathcal F}}
\nc{\calfr}{{{\mathcal F}^{\,r}}} \nc{\calfo}{{\mathcal F}^0}
\nc{\calfro}{{\mathcal F}^{\,r,0}} \nc{\oF}{\overline{F}}
\nc{\calg}{{\mathcal G}} \nc{\calh}{{\mathcal H}}
\nc{\cali}{{\mathcal I}} \nc{\calj}{{\mathcal J}}
\nc{\call}{{\mathcal L}} \nc{\calm}{{\mathcal M}}
\nc{\caln}{{\mathcal N}} \nc{\calo}{{\mathcal O}}
\nc{\calp}{{\mathcal P}} \nc{\calq}{{\mathcal Q}} \nc{\calr}{{\mathcal R}}
\nc{\calt}{{\mathcal T}} \nc{\caltr}{{\mathcal T}^{\,r}}
\nc{\calu}{{\mathcal U}} \nc{\calv}{{\mathcal V}}
\nc{\calw}{{\mathcal W}} \nc{\calx}{{\mathcal X}}
\nc{\CA}{\mathcal{A}}

\nc{\fraka}{{\mathfrak a}} \nc{\frakB}{{\mathfrak B}}
\nc{\frakb}{{\mathfrak b}} \nc{\frakd}{{\mathfrak d}}
\nc{\oD}{\overline{D}}
\nc{\frakF}{{\mathfrak F}} \nc{\frakg}{{\mathfrak g}}
\nc{\frakm}{{\mathfrak m}} \nc{\frakM}{{\mathfrak M}}
\nc{\frakMo}{{\mathfrak M}^0} \nc{\frakp}{{\mathfrak p}}
\nc{\frakS}{{\mathfrak S}} \nc{\frakSo}{{\mathfrak S}^0}
\nc{\fraks}{{\mathfrak s}} \nc{\os}{\overline{\fraks}}
\nc{\frakT}{{\mathfrak T}}
\nc{\oT}{\overline{T}}
\nc{\frakX}{{\mathfrak X}} \nc{\frakXo}{{\mathfrak X}^0}
\nc{\frakx}{{\mathbf x}}
\nc{\frakTx}{\frakT}      
\nc{\frakTa}{\frakT^a}        
\nc{\frakTxo}{\frakTx^0}   
\nc{\caltao}{\calt^{a,0}}   
\nc{\ox}{\overline{\frakx}} \nc{\fraky}{{\mathfrak y}}
\nc{\frakz}{{\mathfrak z}} \nc{\oX}{\overline{X}}

\font\cyr=wncyr10

\nc{\al}{\alpha}
\nc{\lam}{\lambda}
\nc{\lr}{\longrightarrow}
\newcommand{\K}{\mathbb {K}}
\newcommand{\A}{\rm A}

\def\c{\star}
\def\la{\langle}
\def\ra{\rangle}


\title{Admissible Poisson  bialgebras}

\author{Jinting Liang}
\address{Department of Mathematics, Michigan State University, East Lansing, MI, 48823, U.S.A }
\email{liangj26@msu.edu}

\author{Jiefeng Liu}
\address{School of Mathematics and Statistics, Northeast Normal University, Changchun 130024,
China} \email{liujf12@126.com}

\author{Chengming Bai}
\address{Chern Institute of Mathematics \& LPMC, Nankai University, Tianjin 300071, China }
\email{baicm@nankai.edu.cn}


\begin{abstract}

An admissible Poisson algebra (or briefly, an adm-Poisson algebra)
gives an equivalent presentation with only one operation for
a Poisson algebra. 
We establish a bialgebra theory for adm-Poisson algebras
independently and systematically, including but beyond the
corresponding results on Poisson bialgebras given in \cite{NB1}.
Explicitly, we introduce the notion of adm-Poisson bialgebras
which are equivalent to Manin triples of adm-Poisson algebras as
well as Poisson bialgebras. The direct correspondence between
adm-Poisson bialgebras with one comultiplication and Poisson
bialgebras with one cocommutative and one anti-cocommutative
comultiplications generalizes and illustrates the
polarization-depolarization process
 in the context of bialgebras. The study of a special class of
adm-Poisson bialgebras which include the known coboundary Poisson
bialgebras in \cite{NB1} as a proper subclass in general,
illustrating an advantage in terms of the presentation with one
operation, leads to the introduction of adm-Poisson Yang-Baxter
equation in an adm-Poisson algebra. 
It is an unexpected consequence that both the adm-Poisson
Yang-Baxter equation and the associative Yang-Baxter equation have
the same form and thus it motivates and simplifies the involved
study from the study of the associative Yang-Baxter equation,
which is another advantage in terms of the presentation with one
operation.
 A skew-symmetric solution of adm-Poisson Yang-Baxter equation
gives an adm-Poisson bialgebra. Finally the notions of an
$\mathcal O$-operator of an adm-Poisson algebra and a
pre-adm-Poisson algebra are introduced to construct skew-symmetric
solutions of adm-Poisson Yang-Baxter equation and hence adm-Poisson bialgebras.
Note that a
pre-adm-Poisson algebra gives an equivalent presentation for a
pre-Poisson algebra introduced by Aguiar.

\end{abstract}

\subjclass[2010]{16T10, 16T25, 17B63}

\keywords{Poisson algebra, bialgebra, classical Yang-Baxter
equation, $\mathcal{O}$-operator}

\maketitle

\tableofcontents

\numberwithin{equation}{section}

\tableofcontents
\numberwithin{equation}{section}
\allowdisplaybreaks

\section{Introduction}




A Poisson algebra whose name comes from the French mathematician
Sim\'{e}on Poisson, is an algebra with a Lie algebra structure and
a commutative associative algebra structure which are entwined by
the Leibniz rule. Poisson algebras appear in a lot of fields such
as Poisson geometry \cite{Wei77,Vaisman1}, classical and quantum mechanics \cite{Arn78,Dirac64,OdA},  algebraic geometry \cite{GK04,Pol97}, quantization theory \cite{Hue90,Kon03} and quantum groups \cite{CP1,Dr87}.



\begin{defi}{\rm(\cite{Li77,Wei77})}
Let $P$ be a vector space equipped with two bilinear operations $
[\;,\;],\circ :P\otimes P\to P$. $(P,[\;,\;],\circ)$ is called a
\textbf{Poisson algebra} if $(P,[\;,\;])$ is a Lie algebra,
$(P,\circ)$ is a commutative associative algebra and
\begin{equation}\label{leibniz rule}
[x,y\circ z]=[x,y]\circ z+y\circ [x,z],\quad \forall x,y,z\in P.
\end{equation}
\end{defi}

It is remarkable that there is an equivalent presentation for a
Poisson algebra $(P,[\;,\;],\circ)$ with one operation as follows.

\begin{defi} {\rm (\cite{GR06})}
Let $P$ be a vector space equipped with one bilinear operation
$\c:P\otimes P\rightarrow P$. We call $(P,\c)$ an {\bf admissible
Poisson algebra} if the following equation holds:
\begin{equation}
(x\c y)\c z=x\c(y\c z)-\frac{1}{3}(-x\c(z\c y)+z\c(x\c y)+y\c(x\c
z)-y\c(z\c x)),\;\;\forall x,y,z\in P.\label{c1}
\end{equation}
\end{defi}

\begin{pro}{\rm (\cite{MR06})} If $(P,[\;,\;],\circ)$ is a
Poisson algebra, then $(P,\c)$ is an admissible Poisson algebra,
which is called the {\bf corresponding admissible Poisson
algebra}, where the multiplication $\c$ is defined by
\begin{equation}\label{eq:PA-sPA}
x\star y=x\circ y+[x,y],\;\;\forall x,y\in P.\end{equation}
Conversely, if $(P,\c)$ is an admissible Poisson algebra, then
$(P,[\;,\;],\circ)$ is a Poisson algebra, which is called the {\bf
corresponding Poisson algebra}, where the multiplication $\circ$
and the bracket operation $[\;,\;]$ are respectively defined by
\begin{equation}\label{eq:sPA-PA}
 x\circ
y=\frac{1}{2}(x\c y+y\c x),\;\;\; [x,y]=\frac{1}{2}(x\c y-y\c
x),\;\;\forall x,y\in P.
\end{equation}
\end{pro}

The structure of an admissible Poisson algebra was given
explicitly in \cite{MR06} by a polarization-depolarization trick
which was also independently employed in \cite{LiLo}, that is,
polarization interprets structures with one operation as
structures with one commutative and anticommutative operations,
whereas conversely depolarization interprets structures with one
commutative and anticommutative operations as structures with one
operation (\cite{LV,OPV}). There has not been a formal notion for
such a structure until the present notion was given in
\cite{GR06}. To avoid confusion, we call an admissible Poisson
algebra briefly  an {\bf adm-Poisson algebra},  whereas an
ordinary Poisson algebra is still called a {\bf Poisson algebra}.
Moreover, same presentations for Poisson superalgebras were given
in \cite{Rem12} and there is a work on nonassociative Poisson
algebras by presenting an approach of the nonassociative
admissible Poisson algebra which is called weakly associative
algebra in \cite{Rem21}.

Although the two presentations are equivalent, as pointed out in
\cite{MR06}, ``this change of perspective might sometimes lead to
new insights and results", and hence there might be respective
advantages when we study some properties or apply them in other
topics in terms of different presentations.

In this paper we establish a bialgebra theory for adm-Poisson
algebras. We would like to point out that there is a bialgebra
theory for Poisson algebras in terms of the usual presentation in
\cite{NB1}. Note that there is also a bialgebra theory for
noncommutative Poison algebras (\cite{LBS}). The study of the
bialgebra theory in terms of adm-Poisson algebras is completely
independent and systematic, including but beyond the corresponding
results given in \cite{NB1}.

Explicitly, we still take a similar approach as of the study on
Lie bialgebras (\cite{CP1,Dr}), that is, the compatibility
condition is still decided by an analogue of a Manin triple of Lie
algebras, which we call a Manin triple of adm-Poisson algebras.
The notion of an adm-Poisson bialgebra is thus introduced as an
equivalent structure of a Manin triple of adm-Poisson algebras,
which is interpreted in terms of matched pairs of adm-Poisson
algebras. These results are parallel to their counterparts in
terms of the usual presentation of Poisson algebras and hence an
adm-Poisson bialgebra is exactly a Poisson bialgebra given in
\cite{NB1}. Furthermore, there is a direct correspondence between
adm-Poisson bialgebras with one comultiplication and Poisson
bialgebras with one cocommutative and one anti-cocommutative
comultiplications, generalizing and illustrating the
polarization-depolarization process in the context of bialgebras.

The results in terms of the presentation with one operation are
``beyond" the usual presentation appear in the study of a special
class of adm-Poisson bialgebras which is similar to the study of coboundary Lie
bialgebras for Lie algebras (\cite{CP1,Dr}) or coboundary infinitesimal bialgebras for associative algebras
(\cite{Aguiar1,Bai2}).
The corresponding Poisson bialgebras of such adm-Poisson
bialgebras include the coboundary Poisson bialgebras introduced in
\cite{NB1} as a proper subclass in general, that is, there might
exist some new examples of non-coboundary Poisson bialgebras from
such adm-Poisson bialgebras, although we have not obtained such an
example explicitly. This feature illustrates an advantage in terms
of the presentation with one operation, as well as might help
demonstrate more implications in some applications.

The study of such adm-Poisson bialgebras also leads to the
introduction of adm-Poisson Yang-Baxter equation in an adm-Poisson
equation which is an analogue of the classical Yang-Baxter
equation in a Lie algebra or the associative Yang-Baxter equation
in an associative algebra. It is also interesting to see that the
set of solutions of adm-Poisson Yang-Baxter equation in an
adm-Poisson algebra includes the set of solutions of Poisson
Yang-Baxter equation given in \cite{NB1} in the corresponding
Poisson algebra as a subset in general, whereas the two sets
coincide under certain conditions including the skew-symmetric
case. Moreover, a skew-symmetric solution of adm-Poisson
Yang-Baxter equation gives an  adm-Poisson bialgebra.

There is an unexpected consequence that both the adm-Poisson
Yang-Baxter equation and the associative Yang-Baxter equation
 have the same form. It is partly due to the
fact that both adm-Poisson and associative algebras have the same
forms of dual representations. Therefore some properties of
adm-Poisson Yang-Baxter equation can be obtained directly from the
corresponding ones of associative Yang-Baxter equation.
Consequently, the study on adm-Poisson algebras involving them
would get more motivations or simplicities from the study of the
associative Yang-Baxter equation, which is another advantage in
terms of the presentation with one operation.

In particular, as for the study on the associative Yang-Baxter
equation (\cite{Aguiar1,Bai2,BGN1}), in order to obtain
skew-symmetric solutions of adm-Poisson Yang-Baxter equation, we
introduce the notion of an $\mathcal O$-operator of an
adm-Poisson algebra which is an analogue of an $\mathcal
O$-operator of a Lie algebra introduced by Kupershmidt in
\cite{Kupershmidt} as a
 natural generalization of the classical Yang-Baxter equation in a Lie algebra (also see \cite{B1}), and a pre-adm-Poisson
 algebra as an analogue of a dendriform algebra introduced by
 Loday (\cite{Loday}).
The former gives a construction of skew-symmetric solutions of
adm-Poisson Yang-Baxter equation  in a semi-direct product
adm-Poisson algebra, whereas the latter gives a representation of
the sub-adjacent adm-Poisson algebra such that the identity is a
natural $\mathcal O$-operator associated to it. Therefore a
construction of skew-symmetric solutions of adm-Poisson
Yang-Baxter equation and hence adm-Poisson bialgebras from
pre-adm-Poisson algebras is given. Note that a pre-adm-Poisson
algebra gives an equivalent presentation for a pre-Poisson algebra
introduced by Aguiar in \cite{Aguiar2}.

This paper is organized as follows. In Section~\ref{pre}, we
recall some facts on Poisson bialgebras and coboundary Poisson
bialgebras given in \cite{NB1}. In Section~\ref{sp}, we introduce
the notions of representations and matched pairs of adm-Poisson
algebras. The dual representation of a representation of an
adm-Poisson algebra is also given. In Section 4, the notions of
Manin triples of adm-Poisson algebras and adm-Poisson bialgebras
are introduced. Their equivalence is interpreted in terms of
matched pairs of adm-Poisson algebras. We also give a direct
correspondence between adm-Poisson bialgebras and Poisson
bialgebras. In Section 5, we consider a special class of
adm-Poisson bialgebras. 
They include coboundary Poisson bialgebras as a proper subclass.
This study also leads to the introduction of adm-Poisson
Yang-Baxter equation whose skew-symmetric solutions give
adm-Poisson bialgebras. In Section 6,  we introduce the notions of
$\mathcal{O}$-operators of adm-Poisson algebras and
pre-adm-Poisson algebras, and give
constructions of skew-symmetric solutions of adm-Poisson
Yang-Baxter equation from these structures.


We adopt the following conventions and notations.

\begin{enumerate}

\item[(1)] Let $(A,\diamond)$ be a vector space with a bilinear
operation $\diamond :A\otimes A\to A$. Let $r=\begin{matrix}\sum_i
x_i\otimes y_i\in A\otimes A\end{matrix}$. Set
      \begin{equation}
        r_{12}=\sum_i x_i\otimes y_i\otimes 1,\; r_{13}=\sum_i x_i\otimes 1\otimes y_i, \; r_{23}=\sum_i 1\otimes x_i\otimes y_i,
      \end{equation}
     where $1$ is the unit if $(A,\diamond)$ is unital or a symbol playing a similar role as the unit for the non-unital cases. The operation between two $r$s is given in an obvious way. For
     example,
        \begin{eqnarray}
        &&r_{12}\diamond r_{13}=\sum_{i,j}x_i\diamond x_j\otimes y_i\otimes y_j,\quad
        r_{13}\diamond r_{23}=\sum_{i,j}x_i\otimes x_j\otimes y_i\diamond
        y_j,\nonumber\\
       && r_{23}\diamond r_{12}=\sum_{i,j}x_i\otimes x_j\diamond y_i\otimes y_j.
        \end{eqnarray}

\item[(2)] Let $V$ be a vector space. Let $\tau:V\otimes V\to
V\otimes V$ be the twisting operator defined as
    \begin{equation}
    \tau(u\otimes v)=v\otimes u, \quad \forall u,v \in V.
    \end{equation}
     \item[(3)] Let $V$ be a vector space. For an $r\in V\otimes V$, the linear map $r^\sharp:V^*\longrightarrow V$ is given by
$$\langle r^\sharp(u^*),v^*\rangle=\langle r, u^*\otimes v^*\rangle,\quad \forall u^*,v^*\in P^*.$$
We say that $r\in V\otimes V$ is nondegenerate, if the linear map $r^\sharp$ is an isomorphism.

\item[(4)] Let $V_1,V_2$ be two vector spaces and $T:V_1\to V_2$
be a linear map. Denote the dual map by $T^*:V_2^*\to
V_1^*$, defined by
    \begin{equation}
    \langle v_1,T^*(v_2^*)\rangle =\langle T(v_1),v_2^*\rangle,\quad \forall v_1 \in V_1, v_2^* \in V_2^*.
    \end{equation}
\item[(5)] Let $V$ be a vector space and $A$ be a vector space (usually with some bilinear operations). For a linear map $\rho:A\to
{\rm End}_{\mathbb F}(V)$, define a linear map $\rho^*:A\to {\rm
End}_{\mathbb F}(V^*)$ by
    \begin{equation}
    \langle \rho^*(x)v^*,u\rangle=-\langle v^*,\rho(x)u\rangle, \quad\forall x\in A,u\in V,v^* \in V^*.
    \end{equation}

\end{enumerate}

Throughout this paper, all vector spaces are finite-dimensional
over a base field $\mathbb F$, 
although many results still hold in the infinite dimension.

\section{Some facts on Poisson bialgebras and coboundary Poisson bialgebras}\label{pre}

In this section, we recall some facts on Poisson bialgebras and coboundary Poisson bialgebras given in
 \cite{NB1}. 

Let $(\mathfrak{g},[\;,\;])$ be a Lie algebra. Let ${\rm ad}(x)$
denote the adjoint operator, that is, ${\rm ad}(x)(y)=[x,y]$ for
all $x,y\in \frak g$. Let ${\rm ad}:\frak g\rightarrow
\End_{\mathbb F}(\frak g)$ with $x\rightarrow {\rm ad}(x)$ be the
adjoint representation of $\frak g$. A {\bf Lie bialgebra}
structure on $\mathfrak{g}$ is a linear map
$\delta:\mathfrak{g}\to \wedge^2\mathfrak{g}$ such that
$\delta^*:\wedge^2\frak g^*\to \frak g^*$ defines a Lie algebra
structure on $\frak g^*$  and $\delta$ is a $1$-cocycle of
$\mathfrak{g}$ with the coefficient in the representation
$(\g\otimes\g;\ad\otimes {\rm id}+{\rm id}\otimes\ad)$, i.e.
\begin{equation}\label{coboundary Lie}
\delta([x,y])=({\rm ad}(x)\otimes{\rm id}+{\rm id}\otimes{\rm ad}(x))\delta(y)-({\rm ad}(y)\otimes{\rm id}+{\rm id}\otimes{\rm ad}(y))\delta(x),\quad \forall x,y\in \mathfrak{g}.
\end{equation}

Let $(A,\circ)$ be an associative algebra. Let $L_\circ(x)$ and
$R_\circ(x)$ be the left and the right multiplication operators
respectively, that is, $L_\circ(x)(y)=R_\circ(y)(x)=x\circ y$ for
all $x,y\in A$. Let $L_\circ,R_\circ:A\rightarrow \End_{\mathbb
F}(A)$ be two linear maps with $x\rightarrow L_\circ(x)$ and $
x\rightarrow R_\circ(x)$ respectively. An {\bf infinitesimal
bialgebra} structure on $A$ is a linear map $\Delta:A\to A\otimes
A$ such that $\Delta^*:A^*\otimes A^*\to A^*$ defines an
associative algebra structure on $A^*$  and $\Delta$ satisfies
\begin{equation}\label{infinitesimal bialg1}
\Delta(x\circ y)=({\rm id}\otimes L_\circ(x))\Delta(y)+(R_\circ(y)\otimes {\rm id})\Delta(x),\;\;\forall x, y\in A.
\end{equation}

\delete{

\liu{The Harrison cohomology theory for a commutative algebra
$(A,\cdot)$ associated to a module $(V;\mu)$ is given as follows.
A permutation $\sigma\in\perm_n$ is called an $(i,n-i)$-unshuffle
if $\sigma(1)<\cdots<\sigma(i)$ and
$\sigma(i+1)<\cdots<\sigma(n)$. The set of all
$(i,n-i)$-unshuffles will be denoted by $\perm_{(i,n-i)}$. Denote
by $C_{\rm H}^n(A,V)$ the space of $n$-cochains such that for all
$f\in\Hom(\otimes^n A,V)$
$$\sum_{\sigma\in\perm_{(i,n-i)}}\sgn(\sigma)f(x_{\sigma^{-1}(1)},\ldots,x_{\sigma^{-1}(i)},x_{\sigma^{-1}(i+1)},\ldots,x_{\sigma^{-1}(n)})=0,\quad 0<i< n,$$
where $x_1,x_2,\ldots,x_n \in A$. The corresponding Harrison coboundary operator
$\dM_{\rm H}:C_{\rm H}^n(A,V)\longrightarrow C_{\rm H}^{n+1}(A,V)$ is given by
\begin{eqnarray*}
\dM_{\rm H}f(x_1,\ldots,x_{n+1})&=&\mu({x_1}) f(x_2,\ldots,x_{n+1})+\sum_{i=1}^n(-1)^if(x_1,\ldots,x_{i-1},x_i\cdot x_{i+1},\ldots,x_{n+1})\\
&&+(-1)^{n+1}\mu({ x_{n+1}})f(x_1,\ldots,x_n)
\end{eqnarray*}
for all $f\in C_{\rm H}^n(A,V)$ and $x_1,x_2,\cdots,x_{n+1}\in A$.}

\liu{Note that $1$-cocycle in Harrison cohomology is given by
$$f(x\cdot y)=\mu(x)f(y)-\mu(y)f(x).$$
When $A$ is a commutative algebra, the Eq. $(\ref{infinitesimal
bialg1})$ is not the $1$-cocycle in Harrison cohomology.  The Eq.
$(\ref{infinitesimal bialg1})$ is just the $1$-cocycle in
Hochschild cohomology for the commutative algebra $A$. Thus the
commutative associative bialgebras don't have a cohomology
explanation in the context of Harrison cohomology theory.
Similarly, the commutative Poisson bialgebras don't have a
cohomology explanation in the context of commutative Poisson
cohomology theory. }}

\begin{defi}
Let $(P,[\;,\;],\circ)$ be a Poisson algebra. Let
$\delta:P\to \wedge^2P$ and $\Delta:P\to P\otimes P$ be two linear maps such that
$\delta^*:\wedge^2 P^*\to P^*$ defines a Lie algebra structure
on $P^*$, $\Delta^*:P^*\otimes P^*\to P^*$ defines a commutative
associative algebra structure on $P^*$, and they satisfy the
following compatible condition:
    \begin{equation}
    ({\rm id}\otimes \Delta)\delta(x)=(\delta\otimes {\rm id})\Delta(x)+(\tau\otimes {\rm id})({\rm id}\otimes \delta)\Delta(x),\quad \forall x\in P.
    \end{equation}
If in addition, $(P,[\;,\;],\delta)$ is a Lie bialgebra,
$(P,\circ,\Delta)$ is an infinitesimal bialgebra, and $\delta$ and
$\Delta$ are compatible in the following sense
 \begin{eqnarray}
 \label{c5.1}\delta(x\circ y)&=&(L_\circ(x)\otimes {\rm id})\delta(y)+(L_\circ(y)\otimes {\rm id})\delta(x)+({\rm id}\otimes{\rm ad}(x))\Delta(y)\\
 \nonumber&&+({\rm id}\otimes {\rm ad}(y))\Delta(x),\\
 \label{c5.2}\Delta([x,y])&=&({\rm ad}(x)\otimes {\rm id}+{\rm id}\otimes{\rm ad}(x))\Delta(y)+(L_\circ(y)\otimes {\rm id}-{\rm id}\otimes L_\circ(y))\delta(x)
 \end{eqnarray}
for all $x,y\in P$, then $(P,[\;,\;],\circ,\delta,\Delta)$ is
called a \textup{\textbf{Poisson bialgebra}}.
\end{defi}


A Lie bialgebra $(\mathfrak{g},\delta)$ is called \textbf{coboundary} if 
there exists an $r\in \mathfrak{g}\otimes \mathfrak{g}$ such that
\begin{equation}\label{coboundary Lie2}
\delta(x)=({\rm ad}(x)\otimes{\rm id}+{\rm id}\otimes{\rm ad}(x))r,\quad \forall x\in \mathfrak{g}.
\end{equation}
In this case, $\delta$ automatically satisfies Eq. $(\ref{coboundary Lie})$. 

Let $\mathfrak{g}$ be a Lie algebra and $r\in
\mathfrak{g}\otimes\mathfrak{g}$. The linear map $\delta$ defined
by Eq. $(\ref{coboundary Lie2})$ makes $(\mathfrak{g},\delta)$
become a Lie bialgebra if and only if the following conditions are
satisfied:
\begin{align}
&({\rm ad}(x)\otimes{\rm id}+{\rm id}\otimes {\rm ad}(x))(r+\tau(r))=0,\\
&({\rm ad}(x)\otimes{\rm id}\otimes{\rm id}+{\rm id}\otimes {\rm ad}(x)\otimes{\rm id}+{\rm id}\otimes{\rm id}\otimes {\rm ad}(x))\textup{\textbf{C}}(r)=0
\end{align}
for all $x\in \mathfrak{g}$, where
$\textup{\textbf{C}}(r)=[r_{23}, r_{12}]+[r_{23},r_{13}]+[r_{13},r_{12}]$.

In particular, the following equation:
\begin{equation}
\textup{\textbf{C}}(r)=[r_{23}, r_{12}]+[r_{23},r_{13}]+[r_{13},r_{12}]=0
\end{equation}
is called \textbf{classical Yang-Baxter equation (CYBE)}.

An infinitesimal bialgebra $(A,\Delta)$ is called \textbf{coboundary} if there exists an $r\in A\otimes A$ such that
\begin{equation}\label{coboundary infini}
\Delta(x)=({\rm id}\otimes L_\circ(x)- R_\circ(x)\otimes {\rm
id})r,\quad \forall x\in A.
\end{equation}
In this case, $\Delta$ automatically satisfies Eq. $(\ref{infinitesimal bialg1})$. 

Let $(A,\circ)$  be a commutative associative algebra and $r\in
A\otimes A$. The linear map $\Delta$ defined by Eq.
$(\ref{coboundary infini})$ makes $(A,\Delta)$ become an
infinitesimal bialgebra such that $\Delta^*$ defines a commutative
associative algebra structure on $A^*$ if and only if the
following conditions are satisfied:
\begin{align}
&(L_{\circ}(x)\otimes {\rm id}-{\rm id}\otimes L_{\circ}(x))(r+\tau(r))=0,\\
&(L_{\circ}(x)\otimes {\rm id}\otimes {\rm id}-{\rm id}\otimes{\rm id}\otimes L_{\circ}(x))\textup{\textbf{A}}(r)=0
\end{align}
for all $x\in A$, where $\textup{\textbf{A}}(r)=r_{23}\circ r_{12}-r_{13}\circ r_{23}-r_{12}\circ
r_{13}$.

In particular, the following equation:
\begin{equation}\label{eq:AYBE}
\textup{\textbf{A}}(r)=r_{23}\circ r_{12}-r_{13}\circ r_{23}-r_{12}\circ
r_{13}=0
\end{equation}
is called \textbf{associative Yang-Baxter equation (AYBE)}.

\begin{defi}
A Poisson bialgebra $(P,[\;,\;],\circ,\delta,\Delta)$ is called
\textbf{coboundary} if $\delta$ and $\Delta$ satisfy
\begin{eqnarray}
 \delta(x)&=&({\rm ad}(x)\otimes {\rm id}+{\rm id}\otimes{\rm ad}(x))r,\label{eq:PB1}\\
   \Delta(x)&=&({\rm id}\otimes L_{\circ}(x)-L_{\circ}(x)\otimes{\rm id})r\label{eq:PB2}
\end{eqnarray}
for all $x\in P$ and some $r\in P\otimes P$. We sometimes denote
it by $(P,[\;,\;],\circ,r)$.
\end{defi}



\begin{thm}\label{cLie-Poissonbialg}
Let $(P,[\;,\;],\circ)$ be a Poisson algebra and $r\in P\otimes
P$. Let $\delta:P\to \wedge^2P$ and $\Delta:P\to P\otimes P$ be
two linear maps defined by Eqs. $(\ref{eq:PB1})$ and
$(\ref{eq:PB2})$ respectively. Then
$(P,[\;,\;],\circ,\delta,\Delta)$ is a Poisson bialgebra if and
only if for all $x\in P$, the following conditions are satisfied:
\begin{itemize}
    \item[\textup{(1)}]
    $({\rm ad}(x)\otimes {\rm id}+{\rm id}\otimes{\rm ad}(x))(r+\tau(r))=0$;
    \item[\textup{(2)}]
    $(L_{\circ}(x)\otimes {\rm id}-{\rm id}\otimes L_{\circ}(x))(r+\tau(r))=0$;
    \item[\textup{(3)}]
    $({\rm ad}(x)\otimes {\rm id}\otimes {\rm id}+{\rm id}\otimes{\rm ad}(x)\otimes {\rm id}+{\rm id}\otimes {\rm id}\otimes {\rm ad}(x))\textup{\textbf{C}}(r)=0$;
    \item[\textup{(4)}]
    $(L_{\circ}(x)\otimes {\rm id}\otimes {\rm id}-{\rm id}\otimes{\rm id}\otimes L_{\circ}(x))\textup{\textbf{A}}(r)=0$;
    \item[\textup{(5)}]
    $({\rm ad}(x)\otimes {\rm id}\otimes {\rm id})\textup{\textbf{A}}(r)-({\rm id}\otimes L_{\circ}(x)\otimes{\rm id}
    -{\rm id}\otimes{\rm id}\otimes L_{\circ}(x))\textup{\textbf{C}}(r)=0$.
\end{itemize}
\end{thm}

\begin{defi}
Let $(P,[\;,\;],\circ)$ be a Poisson algebra and $r\in P\otimes
P$. $r$ is a solution of {\bf Poisson Yang-Baxter equation (PYBE)}
in $(P,[\;,\;],\circ)$  if $r$ is a solution of both CYBE and
AYBE, that is, $\textup{\textbf{C}}(r)=\textup{\textbf{A}}(r)=0$.

\end{defi}

\begin{cor}
Let $(P,[\;,\;],\circ)$ be a Poisson algebra and $r\in P\otimes
P$. Let $\delta:P\to \wedge^2P$ and $\Delta:P\to P\otimes P$ be
two linear maps defined by Eqs. $(\ref{eq:PB1})$ and
$(\ref{eq:PB2})$ respectively. If $r$ is a skew-symmetric solution
of PYBE 
 in $(P,[\;,\;],\circ)$, then
$(P,[\;,\;],\circ,\delta,\Delta)$ is a Poisson bialgebra.
\end{cor}

\section{Representations  and matched pairs of adm-Poisson algebras \label{sp}}
In this section, we introduce the notions of representations and
matched pairs of adm-Poisson algebras and then give some
properties.

\emptycomment{Let $(P,\c)$ be an adm-Poisson algebra. For
convenience,  we denote the left and right multiplication
operators by $L(x)$ and $R(x)$ respectively without confusion,
 that is, $L(x)y=R(y)x=x\c y$ for all $x,y\in P$. Their relation with the corresponding operators of the corresponding Poisson algebra
 $(P,[\;,\;],\circ)$ is \liu{The relations between the operators $L$ and $R$ in adm-Poisson algebra and the operators $L_\circ$ and ${\rm ad}$ in the corresponding Poisson algebra are given by}
\begin{equation}L=L_\circ+{\rm ad},\quad R=L_\circ-{\rm ad}.\end{equation}
Let $L,R:P\rightarrow {\rm End}(P)$ be two linear maps with
$a\rightarrow L(a)$ and $a\rightarrow R(a)$ respectively.}

\begin{defi}\label{definition module}
Let $(P,\star)$ be an adm-Poisson algebra and $V$ be a vector
space. Let ${\frkl},\frkr:P\to {\rm End}_{\mathbb F}(V)$ be two linear maps.
The triple $(\frkl,\frkr,V)$ is called a \textup{\textbf{representation}}
of $(P,\star )$ if
\begin{eqnarray}
\frkl(x\star y)=\frkl(x)\frkl(y)-\frac{1}{3}\big(-\frkl(x)\frkl(y)+\frkr(x\c y)+\frkl(y)\frkl(x)-\frkl(y)\frkl(x)\big),\label{c2} \\
\frkr(y)\frkl(x)=\frkl(x)\frkr(y)-\frac{1}{3}\big(-\frkl(x)\frkl(y)+\frkl(y)\frkl(x)+\frkr(x\star y)-\frkr(y\star x)\big),\label{c3}\\
\frkr(y)\frkr(x)=\frkr(x\star y)-\frac{1}{3}\big(-\frkr(y\star x)+\frkl(y)\frkl(x)+\frkl(x)\frkr(y)-\frkl(x)\frkl(y)\big)\label{c4}
\end{eqnarray}
for all $x,y\in P$. Two representations $(\frkl_1,\frkr_1,V_1)$ and
$(\frkl_2,\frkr_2,V_2)$ of an adm-Poisson algebra $P$ are called {\bf
equivalent} if there exists an isomorphism $\varphi:V_1\rightarrow
V_2$ satisfying
\begin{equation}\label{lem:rep1}
\varphi \frkl_1(x) =\frkl_2(x)\varphi,\quad \varphi
\frkr_1(x)=\frkr_2(x)\varphi,\;\;\forall x\in P.
\end{equation}
\end{defi}

\begin{lem}\label{lem:rep-property1}
  Let $(\frkl,\frkr,V)$ be a representation of an adm-Poisson algebra $(P,\star)$.
  Then the following equation holds:
  \begin{equation}\label{eq:rep-property1}
    \frkl(x\star y)+\frkr(x) \frkr(y)=\frkl(x)\frkl(y)+\frkr(y\star x),\;\;\forall x,y\in
    P.
  \end{equation}
\end{lem}
\begin{proof}
  On the one hand, by Eq. \eqref{c2}, we have
  $$\frac{1}{3}\big(-\frkl(x)\frkr(y)+\frkr(x\c y)+\frkl(y)\frkl(x)-\frkl(y)\frkr(x)\big)=\frkl(x)\frkl(y)-\frkl(x\star y),\;\;\forall x,y\in P.$$
 On the other hand, by Eq. \eqref{c4}, we have
  $$\frac{1}{3}\big(-\frkl(x)\frkr(y)+\frkr(x\c y)+\frkl(y)\frkl(x)-\frkl(y)\frkr(x)\big)=\frkr(x)\frkr(y)-\frkr(y\star x),\;\;\forall x,y\in P.$$
Hence Eq. \eqref{eq:rep-property1} holds.
\end{proof}

\begin{pro}\label{pro:semi}
Let $(P,\star)$ be an adm-Poisson algebra. Let $V$ be a vector
space and $\frkl,\frkr:P\to {\rm End}_{\mathbb
F}(V)$ be two linear maps. Define a
bilinear operation $\star_{\frkl,\frkr}:(P\oplus V)\otimes (P\oplus V)\to
P\oplus V$ on $P\oplus V$ by
    $$(x+u)\star_{\frkl,\frkr}(y+v)=x\star y+\frkl(x)v+\frkr(y)u, \quad\forall x,y\in P,u,v\in V.$$
Then $(\frkl,\frkr,V)$ is a representation of $(P,\star)$ if and only if
$(P\oplus V,\star_{\frkl,\frkr})$ is an adm-Poisson algebra, which is
called the {\bf semi-direct product} of $P$ by $V$ and denoted by
$P\ltimes_{\frkl,\frkr} V$ or simply $P\ltimes V$.
\end{pro}

\begin{proof}

It is due to \cite{Sch} with a straightforward proof.
\end{proof}

\begin{rmk}\label{rmk:repre} Recall that a representation of a Poisson algebra
$(P,[\;,\;],\circ)$ is a triple  $(S_{[\;,\;]},S_\circ,V)$ such
that $(S_{[\;,\;]},V)$ is a representation of the Lie algebra
$(P,[\;,\;])$ and $(S_\circ,V)$ is a representation of the
commutative associative algebra $(P,\circ)$ satisfying some
compatible conditions (\cite{NB1}). Let $(\frkl,\frkr,V)$ be a
representation of an adm-Poisson algebra $(P,\star_P)$. Then
$\left(\frac{1}{2}(\frkl-\frkr),\frac{1}{2}(\frkl+\frkr),V\right)$ is a
representation of the corresponding Poisson algebra
$(P,[\;,\;],\circ)$. Conversely, if $(S_{[\;,\;]},S_\circ,V)$ is a
representation of a Poisson algebra $(P,[\;,\;],\circ)$, then
$(S_{[\;,\;]}+S_\circ,S_\circ-S_{[\;,\;]},V)$ is a representation
of the corresponding adm-Poisson algebra $(P,\star_P)$. Thus the
representations of an adm-Poisson algebra are in a one-to-one
correspondence with the representations of the corresponding
Poisson algebra.
\end{rmk}

\begin{pro}
Let $(\frkl,\frkr,V)$ be a representation of an adm-Poisson algebra
$(P,\star)$. Then  $(-\frkr^*$, $-\frkl^*$, $V^*)$ is a representation of
$(P,\star)$. We call it the {\bf dual representation} of
$(\frkl,\frkr,V)$.
\end{pro}

\delete{
\begin{proof}
It follows from a direct calculation. Here we only give a proof of
the first condition Eq. $(\ref{c2})$ in Definition \ref{definition
module} for $(-\frkr^*,-\frkl^*,V^*)$ as an example. The proofs for the
other two conditions are similar. To this end, one only need to
check that for all $a,b\in P$, the following equation holds:
$$\frkr^*(a\c b)+\frkr^*(a)\frkr^*(b)+\frac{1}{3}(\frkr^*(a)\frkl^*(b)+\frkl^*(a\c b)-\frkr^*(b)\frkr^*(a)+\frkr^*(b)\frkl^*(a))=0.$$
Note that the properties of representations of an adm-Poisson
algebra give
$$\frkl(a\c b)+\frkr(a)\frkr(b)=\frkl(a)\frkl(b)+\frkr(b\c a),\quad \forall a,b\in P.$$
Then for arbitrary $u\in V,\ v^*\in V^*$, we have
\begin{align*}
    &\langle[\frkr^*(a\c b)+\frkr^*(a)\frkr^*(b)+\frac{1}{3}(\frkr^*(a)\frkl^*(b)+\frkl^*(a\c b)-\frkr^*(b)\frkr^*(a)+\frkr^*(b)\frkl^*(a))]v^*,u\rangle\\
    =&\langle v^*,[-\frkr(a\c b)+\frkr(b)\frkr(a)+\frac{1}{3}(\frkl(b)\frkr(a)+\frkl(a)\frkr(b)-\frkl(a\c b)-\frkr(a)\frkr(b)]u\rangle\\
    =&\langle v^*,[-\frkr(a\c b)+\frkr(b)\frkr(a)+\frac{1}{3}(\frkl(b)\frkr(a)+\frkl(a)\frkr(b)-\frkr(b\c a)-\frkl(a)\frkl(b)]u\rangle\\
    =&0,
\end{align*}
hence the conclusion follows.
\end{proof}}

\begin{proof}
  We only show that Eq. $(\ref{c2})$ holds for $(-\frkr^*,-\frkl^*,V^*)$ as an example. The proofs for the holding of Eqs. $(\ref{c3})$ and $(\ref{c4})$ are similar. Since $(\frkl,\frkr,V)$ is a representation of the adm-Poisson algebra
$(P,\star)$, by Lemma \ref{lem:rep-property1}, for all $x,y\in
P,u\in V,\ v^*\in V^*$, we have
\begin{eqnarray*}
  &&\langle\big(\frkr^*(x\c y)+\frkr^*(x)\frkr^*(y)+\frac{1}{3}\big(\frkr^*(x)\frkl^*(y)+\frkl^*(x\c y)-\frkr^*(y)\frkr^*(x)+\frkr^*(y)l^*(x)\big)\big)v^*,u\rangle\\
  &=&\langle v^*,\big(-\frkr(x\c y)+\frkr(y)\frkr(x)+\frac{1}{3}\big(\frkl(y)\frkr(x)+\frkl(x)\frkr(y)-\frkl(x\c y)-\frkr(x)\frkr(y)\big)\big)u\rangle\\
  &=&\langle v^*,\big(-\frkr(x\c y)+\frkr(y)\frkr(x)+\frac{1}{3}\big(\frkl(y)\frkr(x)+\frkl(x)\frkr(y)-\frkr(y\c x)-\frkl(x)\frkl(y)\big)\big)u\rangle=0,
\end{eqnarray*}
which implies that Eq. $(\ref{c2})$ holds for $(-\frkr^*,-\frkl^*,V^*)$.
\end{proof}

\begin{rmk}
In fact, the above conclusion can be obtained from
Remark~\ref{rmk:repre} as follows. Note that for a representation
$(S_{[\;,\;]},S_\circ,V)$ of a Poisson algebra
$(P,[\;,\;],\circ)$, $(S_{[\;,\;]}^*,-S_\circ^*,V)$ is a
representation of $(P,[\;,\;],\circ)$ which is the dual
representation of $(S_{[\;,\;]},S_\circ,V)$. Therefore for a
representation $(\frkl,\frkr,V)$ of an adm-Poisson algebra $(P,\star)$,
$\left(\frac{1}{2}(\frkl-\frkr),\frac{1}{2}(\frkl+\frkr),V\right)$ is a
representation of the corresponding Poisson algebra
$(P,[\;,\;],\circ)$ whose dual representation is
$\left(\frac{1}{2}(\frkl^*-\frkr^*),-\frac{1}{2}(\frkl^*+\frkr^*),V^*\right)$,
which in turn gives a representation $(-\frkr^*,-\frkl^*,V^*)$ of the
original adm-Poisson algebra $(P,\star)$.
\end{rmk}

\begin{rmk}\label{rmk:same}
Note that for a representation $(\frkl,\frkr,V)$ of an associative algebra
in the sense of bimodules, the dual representation is also
$(-\frkr^*,-\frkl^*,V^*)$. Therefore for both associative and adm-Poisson
algebras, the dual representations in the above sense have the
same form.
\end{rmk}

\begin{ex} Let $(P,\star)$ be an adm-Poisson algebra. Let
$L(x)$ and $R(x)$ denote the left and right multiplication
operators, respectively, that is, $L(x)y=R(y)x =x\star y$ for all
$x,y\in P$. Let $L,R:P\rightarrow \End_{\mathbb F}(P)$ be two
linear maps with $x\rightarrow L(x)$ and $ x\rightarrow R(x)$
respectively. Then $(L,R,P)$ is a representation of $(P,\star)$,
called the {\bf adjoint representation}. Furthermore,
$(-R^*,-L^*,P^*)$ is also a representation of $(P,\star)$.
\end{ex}

The relationship between the adjoint representation $(L,R,P)$ of
an adm-Poisson algebra $(P,\star)$ and the representation $({\rm
ad},L_\circ,P)$ of the corresponding Poisson algebra
$(P,[\;,\;],\circ)$ is given by
\begin{equation}L=L_\circ+{\rm ad},\quad R=L_\circ-{\rm ad}.\end{equation}

Letting $(\frkl,\frkr,V)=(L,R, P)$ in Lemma
\ref{lem:rep-property1}, we have the following known conclusion.
\begin{cor}{\rm(\cite{GR06})}
  Let $(P,\star)$ be an adm-Poisson algebra. Then we have
  \begin{equation}\label{eq:adm-Poisson property1}
   (x\star y) \star z -x\star (y\star z)=z\star(y\star x)-(z\star y )\star
   x,\;\;\forall x,y,z\in P.
  \end{equation}
\end{cor}

\begin{defi}\label{Defmatch}
Let $(P_1,\star_1)$ and $(P_2,\star_2)$ be two adm-Poisson
algebras. Let  $\frkl_1,\frkr_1:P_1\to {\rm End}_{\mathbb F}(P_2)$,
$\frkl_2,\frkr_2:P_2\to {\rm End}_{\mathbb F}(P_1)$ be four linear
maps. $(P_1,P_2,\frkl_1,\frkr_1,$ $\frkl_2,\frkr_2)$ is called
\textbf{\textup{a matched pair of adm-Poisson algebras}} if
 $(\frkl_1,\frkr_1,P_2)$ and $(\frkl_2,\frkr_2,P_1)$ are
representations of $(P_1,\star_1)$ and $(P_2,\star_2)$
respectively, and for all $x,y\in P_1,a,b\in P_2$, the following
equations hold.
\begin{eqnarray}
\frkr_2(a)(x\c_1y)&=&\frkr_2(\frkl_1(y)a)x+x\c_1(\frkr_2(a)y)+\frac{1}{3}\big(\frkr_2(\frkr_1(y)a)x+x\c_1(\frkl_2(a)y) \nonumber \\
&&-\frkl_2(a)(x\c_1y)-y\c_1(\frkr_2(a)x)-\frkr_2(\frkl_1(x)a)y+y\c_1(\frkl_2(a)x)\nonumber \\
&&+\frkr_2(\frkr_1(x)a)y\big),\label{Defmatch1}\\
\frkl_2(a)(x\c_1y)&=&(\frkl_2(a)x)\c_1y+\frkl_2(\frkr_1(x)a)y+\frac{1}{3}\big(-\frkl_2(a)(y\c_1x)+y\c_1(\frkl_2(a)x) \nonumber \\
&&+\frkr_2(\frkr_1(x)a)y+x\c_1(\frkl_2(a)y)+\frkr_2(\frkr_1(y)a)x-x\c_1(\frkr_2(a)y)\nonumber \\
&&-\frkr_2(\frkl_1(y)a)x\big),\label{Defmatch2}\\
(\frkr_2(a)x)\c_1y&=&-\frkl_2(\frkl_1(x)a)y+x\c_1(\frkl_2(a)y)+\frkr_2(\frkr_1(y)a)x+\frac{1}{3}\big(x\c_1(\frkr_2(a)y)\nonumber \\
&&+\frkr_2(\frkl_1(y)a)x-y\c_1(\frkr_2(a)x)-\frkr_2(\frkl_1(x)a)y-\frkl_2(a)(x\c_1y)\nonumber \\
&&+\frkl_2(a)(y\c_1x)\big),\label{Defmatch3}\\
\frkr_1(x)(a\c_2b)&=&\frkr_1(\frkl_y2(b)x)a+a\c_2(\frkr_1(x)b)+\frac{1}{3}\big(\frkr_1(\frkr_2(b)x)a+a\c_2(\frkl_1(x)b) \nonumber \\
&&-\frkl_1(x)(a\c_2b)-b\c_2(\frkr_1(x)a)-\frkr_1(\frkl_2(a)x)b+b\c_2(\frkl_1(x)a)\nonumber \\
&&+\frkr_1(\frkr_2(a)x)b\big),\label{Defmatch4}\\
\frkl_1(x)(a\c_2b)&=&(\frkl_1(x)a)\c_2b+\frkl_1(\frkr_2(a)x)b+\frac{1}{3}\big(-\frkl_1(x)(b\c_2a)+b\c_2(\frkl_1(x)a) \nonumber \\
&&+\frkr_1(\frkr_2(a)x)b+a\c_2(\frkl_1(x)b)+\frkr_1(r_2(b)x)a-a\c_2(\frkr_1(x)b)\nonumber \\
&&-\frkr_1(\frkl_2(b)x)a\big),\label{Defmatch5}\\
(\frkr_1(x)a)\c_2b&=&-\frkl_1(\frkl_2(a)x)b +a\c_2(\frkl_1(x)b)+\frkr_1(\frkr_2(b)x)a+\frac{1}{3}\big(a\c_2(\frkr_1(x)b)\nonumber \\
&&+\frkr_1(\frkl_2(b)x)a-b\c_2(\frkr_1(x)a)-\frkr_1(\frkl_2(a)x)b-\frkl_1(x)(a\c_2b)\nonumber \\
&&+\frkl_1(x)(b\c_2a)\big).\label{Defmatch6}
\end{eqnarray}
\end{defi}


By a straightforward proof, we get the following conclusion.
\begin{pro}\label{pro:matched pair}
Let $(P_1,\star_1)$ and $(P_2,\star_2)$ be two adm-Poisson
algebras, $\frkl_1,\frkr_1:P_1\to {\rm End}_{\mathbb F}(P_2)$,
$\frkl_2,\frkr_2:P_2\to {\rm End}_{\mathbb F}(P_1)$ be four linear
maps. Define a bilinear operation $\star:(P_1\oplus P_2)\otimes
(P_1\oplus P_2)\to P_1\oplus P_2$ on $P_1\oplus P_2$ by
\begin{eqnarray}
 (x+a)\c(y+b)=x\c_1y+\frkr_2(b)x+\frkl_2(a)y+\frkl_1(x)b+\frkr_1(y)a+a\c_2b,
\end{eqnarray}
where $x,y\in P_1,a,b\in P_2.$ Then $(P_1\oplus P_2,\c)$ is an
adm-Poisson algebra if and only if $(P_1,P_2$,
$\frkl_1,\frkr_1,\frkl_2,\frkr_2)$ is a matched pair of adm-Poisson
algebras. We denote this adm-Poisson algebra by
$P_1\bowtie_{\frkl_1,\frkr_1}^{\frkl_2,\frkr_2}P_2$ or simply
$P_1\bowtie P_2$.
\end{pro}

Note that the semi-direct product of an adm-Poisson algebra
$(P,\star)$ by a representation $(\frkl,\frkr,V)$ given in
Proposition \ref{pro:semi} is a special case of the matched pairs
of adm-Poisson algebras in Proposition \ref{pro:matched pair} when
$P_2=V$ is equipped with the zero multiplication, that is, $(P,V,
\frkl,\frkr,0,0)$ is a matched pair of adm-Poisson algebras.



\section{Admissible Poisson bialgebras}

In this section, we introduce the notions of Manin triples of
adm-Poisson algebras and adm-Poisson bialgebras. The equivalence
between them is interpreted in terms of matched pairs of
adm-Poisson algebras.

\begin{defi}
 A bilinear form $\frak B$ on an adm-Poisson algebra $(P,\c)$ is called {\bf invariant} if
\begin{equation}
\frak B(x\c y,z)=\frak B(x, y\c z),\;\;\forall x,y,z\in P.
\end{equation}
\end{defi}

\begin{pro}
Let $(P,\c)$ be an adm-Poisson algebra. If there is a
nondegenerate symmetric invariant bilinear form $\frak B$ on $P$,
then the two representations $(L,R,P)$ and $(-R^*,-L^*,P^*)$ of the
adm-Poisson algebra $(P,\c)$ are equivalent. Conversely, if the
two representations $(L,R,P)$ and $(-R^*,-L^*,P^*)$ of the
adm-Poisson algebra $(P,\c)$ are equivalent, then there exists a
nondegenerate invariant bilinear form $\frak B$ on $P$.
\end{pro}

\begin{proof}
Since $\frak B$ is nondegenerate, there exists a linear
isomorphism $\varphi: P\rightarrow P^*$ defined by
$$\langle \varphi (x), y\rangle =\frak B(x,y),\;\;\forall x, y\in P.$$
For all $x,y,z\in P$, we have
\begin{eqnarray*}
\langle \varphi (L(x) y), z\rangle&=&\frak B(x\c y, z)=\frak B(z\c x, y)=\langle \varphi (y), z\c x\rangle
=\langle -R^*(x)\varphi(y), z\rangle;\\
\langle \varphi (R(x) y), z\rangle&=&\frak B(y\c x, z)=\frak B(y,
x\c z)=\langle \varphi (y), x\c z\rangle =\langle
-L^*(x)\varphi(y), z\rangle.
\end{eqnarray*}
Thus  the two representations $(L,R,P)$ and $(-R^*,-L^*,P^*)$ of the
adm-Poisson algebra $(P,\c)$ are equivalent. The converse can be
proved similarly. We omit the details.
\end{proof}

\begin{defi}
Let $(P,\star_P)$ be an adm-Poisson algebra. Suppose that
$(P^*,\c_{P^*})$ is an adm-Poisson algebra on its dual space
$P^*$. If there exists an adm-Poisson algebra structure on the
direct sum $P\oplus P^*$ of the underlying vector spaces of $P$
and $P^*$ such that $(P,\star_P)$ and $(P^*,\star_{P^*})$ are
adm-Poisson subalgebras and the following symmetric bilinear form
$\mathcal{B}_d$ on $P\oplus P^*$ given by
    \begin{equation}\label{cstandardmanin}
    \mathcal{B}_d(x+a^*,y+b^*)=\langle x,b^*\rangle +\langle a^*,y\rangle,\quad \forall\ x,y\in P,\ a^*,b^*\in P^*
    \end{equation}
is invariant, then $(P\oplus P^*,P,P^*)$ is called a
\textbf{\textup{(standard) Manin triple of adm-Poisson algebras}}
associated to ${\mathcal B}_d$.
\end{defi}

\begin{pro}\label{cstruct1}
Let $(P,\star_P)$ be an adm-Poisson algebra. Suppose
$(P^*,\star_{P^*})$ is an adm-Poisson algebra structure on the
dual space $P^*$. Then $(P\oplus P^*,P,P^*)$ is a standard Manin
triple of adm-Poisson algebras associated to  ${\mathcal B}_d$
defined by Eq.~(\ref{cstandardmanin}) if and only if
$(P,P^*,-R_{P}^*,-L_P^*,$ $-R_{P^*}^*,-L_{P^*}^*)$ is a matched
pair of adm-Poisson algebras.
\end{pro}

\begin{proof}
It follows from the same proof of \cite[Theorem 2.2.1]{Bai2}.
\end{proof}

\begin{thm}\label{thm:corr}
Let $(P,\c_P)$ be an adm-Poisson algebra. Suppose that there is an
adm-Poisson structure $``\c_{P^*}"$ on its dual space $P^*$ given
by a linear map $\alpha^*: P^*\otimes P^*\rightarrow P^*$. Then
$(P,P^*,-R_{P}^*,-L_P^*,$ $-R_{P^*}^*,-L_{P^*}^*)$ is a matched
pair of adm-Poisson algebras if and only if
 $\alpha$ satisfies the following equations
\begin{eqnarray}
    &&\alpha(x\star_P y)-(R_P(y)\otimes {\rm id})\alpha(x)-({\rm id}\otimes L_P(x))\alpha(y)\nonumber\\
    &=&\frac{1}{3}\big((L_P(y)\otimes {\rm id})\alpha(x)-({\rm id}\otimes L_P(y))\alpha(x)+(L_P(x)\otimes {\rm id})\alpha(y)-(R_P(x)\otimes {\rm id})\alpha(y)\nonumber\\
    &&+\tau(-\alpha(x\star_P y)+(L_P(x)\otimes {\rm id})\alpha(y)+(L_P(y)\otimes {\rm id})\alpha(x))\big),\label{Defbi1}\\
    &&\alpha(x\star_P y)-(R_P(y)\otimes {\rm id})\alpha(x)-({\rm id}\otimes L_P(x))\alpha(y)\nonumber \\
    &=&\frac{1}{3}\big(\tau((L_P(x)\otimes {\rm id})\alpha(y)+(L_P(y)\otimes {\rm id})\alpha(x)-(R_P(y)\otimes {\rm id})\alpha(x)-({\rm id}\otimes L_P(x))\alpha(y))\nonumber \\
    &&-\alpha(y\star_P x)+(L_P(x)\otimes {\rm id})\alpha(y)+(L_P(y)\otimes {\rm id})\alpha(x)\big) ,\label{Defbi2}\\
    &&(({\rm id}\otimes R_P(y))\alpha(x)-(L_P(y)\otimes {\rm id})\alpha(x)+\tau\left(({\rm id}\otimes R_P(x))\alpha(y)-(L_P(x)\otimes {\rm id})\alpha(y)\right)\nonumber \\
    &=&\frac{1}{3}\big((R_P(y)\otimes {\rm id})\alpha(x)-({\rm id}\otimes L_P(y))\alpha(x)+({\rm id}\otimes L_P(x))\alpha(y)-(R_P(x)\otimes {\rm id})\alpha(y)\nonumber \\
    &&+\tau(\alpha(y\star_P x)-\alpha(x\star_P y))\big)\label{Defbi3}
\end{eqnarray}
for all $x,y\in P$.
\end{thm}

\begin{proof}
By a direct calculation, in the case that $\frkl_1=-R_P^*,\frkr_1=-L_P^*,
\frkl_2=-R_{P^*}^*$ and $\frkr_2=-L_{P^*}^*$, we have
\begin{eqnarray*}
{\rm Eq.}\ (\ref{Defmatch1})\Longleftrightarrow{\rm Eq.}\ (\ref{Defbi1}),\quad{\rm Eq.}\ (\ref{Defmatch2})\Longleftrightarrow{\rm Eq.}\ (\ref{Defbi2}), \quad{\rm Eq.}\ (\ref{Defmatch3})\Longleftrightarrow{\rm Eq.}\
(\ref{Defbi3}).
\end{eqnarray*}
Thus if $(P,P^*,-R_{P}^*,-L_P^*,$ $-R_{P^*}^*,-L_{P^*}^*)$ is a matched
pair of adm-Poisson algebras, then Eqs.~(\ref{Defbi1})-(\ref{Defbi3}) hold.

Conversely, let $x,y\in P$. By Eq.~(\ref{Defbi2}), we have
\begin{eqnarray}
   &&\tau(\alpha(y\star_P x)-(R_P(x)\otimes {\rm id})\alpha(y)-({\rm id}\otimes L_P(y))\alpha(x))\nonumber \\
    &=&\frac{1}{3}\big((L_P(y)\otimes {\rm id})\alpha(x)+(L_P(x)\otimes {\rm id})\alpha(y)-(R_P(x)\otimes {\rm id})\alpha(y)-({\rm id}\otimes L_P(y))\alpha(x)\nonumber \\
    &&+\tau(-\alpha(x\star_P y)+(L_P(y)\otimes {\rm id})\alpha(x)+(L_P(x)\otimes {\rm id})\alpha(y))\big).\label{eq:bialgebra22}
\end{eqnarray}
By Eqs.~(\ref{Defbi1})-(\ref{eq:bialgebra22}), we obtain
\begin{eqnarray}
  &&\alpha(x\star_P y)-(R_P(y)\otimes {\rm id})\alpha(x)-({\rm id}\otimes L_P(x))\alpha(y)\nonumber\\
  &=&\tau(\alpha(y\star_P x)-(R_P(x)\otimes {\rm id})\alpha(y)-({\rm id}\otimes L_P(y))\alpha(x)).\label{eq:important relation1}
\end{eqnarray}
Interchanging $x$ with $y$ in Eq.~(\ref{Defbi3}), we have
\begin{eqnarray}
 &&(({\rm id}\otimes R_P(x))\alpha(y)-(L_P(x)\otimes {\rm id})\alpha(y)+\tau\left(({\rm id}\otimes R_P(y))\alpha(x)-(L_P(y)\otimes {\rm id})\alpha(x)\right)\nonumber \\
    &=&\frac{1}{3}\big((R_P(x)\otimes {\rm id})\alpha(y)-({\rm id}\otimes L_P(x))\alpha(y)+({\rm id}\otimes L_P(y))\alpha(x)-(R_P(y)\otimes {\rm id})\alpha(x)\nonumber \\
    &&+\tau(\alpha(x\star_P y)-\alpha(y\star_P x))\big).\label{eq:bialgebra33}
\end{eqnarray}
Adding Eq.~(\ref{Defbi3}) and  Eq.~(\ref{eq:bialgebra33})
together, we obtain
\begin{eqnarray}
  &&(L_P(x)\otimes {\rm id})\alpha(y)-({\rm id}\otimes R_P(x))\alpha(y)+\tau((L_P(x)\otimes {\rm id})\alpha(y)+(L_P(y)\otimes {\rm id})\alpha(x))\nonumber\\
  &=&({\rm id}\otimes R_P(y))\alpha(x)-(L_P(y)\otimes {\rm id})\alpha(x)+\tau(({\rm id}\otimes R_P(y) )\alpha(x)+({\rm id}\otimes R_P(x) {\rm id})\alpha(y)).\label{eq:important relation2}
\end{eqnarray}

By Eqs.~\eqref{eq:important relation1} and \eqref{eq:important relation2}, Eq.~(\ref{Defbi1}) implies
\begin{eqnarray*}
    &&\alpha(x\star_P y)-(R_P(y)\otimes {\rm id})\alpha(x)-({\rm id}\otimes L_P(x))\alpha(y)\nonumber\\
    &=&\frac{1}{3}\big(({\rm id}\otimes R_P(x))\alpha(y)-\alpha(y\star_P x)+({\rm id}\otimes R_P(y))\alpha(x)+\tau((-{\rm id}\otimes L_P(x))\alpha(y)\nonumber\\
    &&-(R_P(y)\otimes {\rm id})\alpha(x)+({\rm id}\otimes R_P(x))\alpha(y)+({\rm id}\otimes R_P(y))\alpha(x))\big),
    \end{eqnarray*}
which is equivalent to Eq.~(\ref{Defmatch4}).

By Eqs.~\eqref{eq:important relation1} and \eqref{eq:important relation2}, Eq.~(\ref{Defbi2}) implies
\begin{eqnarray*}
    &&\alpha(y\star_P x)-(R_P(x)\otimes {\rm id})\alpha(y)-({\rm id}\otimes L_P(y))\alpha(x)\nonumber\\
    &=&\frac{1}{3}\big(({\rm id}\otimes R_P(y))\alpha(x)+({\rm id}\otimes R_P(x))\alpha(y)-({\rm id}\otimes L_P(x))\alpha(y)-(R_P(y)\otimes {\rm id})\alpha(x)\nonumber\\
    &&+\tau(-\alpha(y\star_P x)+({\rm id}\otimes R_P(x))\alpha(y)+({\rm id}\otimes R_P(y) )\alpha(x))\big)\label{eq:bialgebra4},
    \end{eqnarray*}
which is equivalent to Eq.~(\ref{Defmatch5}).

By Eq.~\eqref{eq:important relation1}, Eq.~(\ref{Defbi3}) implies
\begin{eqnarray*}
    &&(L_P(x)\otimes {\rm id})\alpha(y)-({\rm id}\otimes R_P(x))\alpha(y)+\tau((L_P(y)\otimes {\rm id})\alpha(x)-({\rm id}\otimes R_P(y))\alpha(x))\nonumber\\
    &=&\frac{1}{3}\big(({\rm id}\otimes L_P(x))\alpha(y)+(R_P(y)\otimes {\rm id})\alpha(x)-\alpha(y\star_P x)\nonumber\\
    &&+\tau(\alpha(y\star_P x)-({\rm id}\otimes L_P(x))\alpha(y)-(R_P(y)\otimes {\rm id} )\alpha(x))\big)\label{eq:bialgebra5},
    \end{eqnarray*}
which is equivalent to Eq.~(\ref{Defmatch6}).

Therefore if Eqs.~(\ref{Defbi1})-(\ref{Defbi3}) hold, then $(P,P^*,-R_{P}^*,-L_P^*,$ $-R_{P^*}^*,-L_{P^*}^*)$ is a matched pair of adm-Poisson algebras.
\end{proof}

\begin{lem}\label{lem:cosp}
Let $P$ be a vector space and $\alpha: P  \rightarrow P\otimes P$
be a linear map. Then the dual map $\alpha^* : P^*\otimes P^*
\rightarrow P^*$ defines an adm-Poisson  algebra structure on
$P^*$ if and only if $\alpha$ satisfies
\begin{eqnarray}
&&({\rm id}\otimes\alpha)\alpha(x)-(\alpha\otimes{\rm
id})\alpha(x)+\frac{1}{3}\big(({\rm id}\otimes \tau)({\rm
id}\otimes \alpha)\alpha(x)-(\tau\otimes{\rm id})({\rm
id}\otimes\alpha)\alpha(x)\nonumber\\
      &&-({\rm id}\otimes\tau)(\tau\otimes{\rm id})({\rm id}\otimes\alpha)\alpha(x)+(\tau\otimes{\rm id})({\rm id}\otimes\tau)({\rm
      id}\otimes\alpha)\alpha(x)\big)=0,\;\;\forall x\in P.\label{eq:coalgebra}
\end{eqnarray}
\end{lem}


\begin{proof}
It is straightforward.
\end{proof}

See \cite{Rem05} for more details on other Lie admissible
coalgebras and their applications.

\begin{defi}\label{adm-Poisson bialgebras}
Let $(P,\star)$ be an adm-Poisson algebra. An {\bf adm-Poisson
bialgebra} structure on $P$ is a linear map $\alpha:P\to P\otimes
P$ such that $\alpha$ satisfies
Eqs.~(\ref{Defbi1})-(\ref{Defbi3}) and (\ref{eq:coalgebra}).
\end{defi}

Combining Proposition~\ref{cstruct1} and Theorem~\ref{thm:corr},
we have the following conclusion.

\begin{thm}
Let $(P,\star_P)$ be an adm-Poisson algebra. Let $\alpha:P\to
P\otimes P$ be a linear map such that $\alpha^*:P^*\otimes P^*\to
P^*$ defines an adm-Poisson algebra structure on $P^*$.  Then the
following conditions are equivalent:
\begin{itemize}
    \item[\textup{(1)}]
     $(P,\c_P,\alpha)$ is an adm-Poisson bialgebra;
     \item[\textup{(2)}]
     $(P,P^*,-R_P^*,-L_P^*,-R_{P^*}^*,-L_{P^*}^*)$ is a matched pair of adm-Poisson algebras;
     \item[\textup{(3)}]
     $(P\oplus P^*,P,P^*)$ is a standard Manin triple of adm-Poisson
     algebras associated to  ${\mathcal B}_d$ defined by
Eq.~(\ref{cstandardmanin}).
\end{itemize}
\end{thm}

There is a one-to-one correspondence between adm-Poisson
bialgebras and Poisson bialgebras.

\begin{pro}\label{cbialgequ}
Let $(P,\star,\alpha)$ be an adm-Poisson bialgebra. Define two
linear maps $\delta,\Delta:P\to P\otimes P$ by
$$\delta=\frac{1}{2}\left(\alpha-\tau\alpha\right),\quad \Delta=\frac{1}{2}\left(\alpha+\tau\alpha\right).$$
Then $(P,[\;,\;],\circ,\delta,\Delta)$ is a Poisson bialgebra.
Conversely, let $(P,[\;,\;],\circ,\delta,\Delta)$ be a Poisson
bialgebra. Define a linear map $\alpha:P\to P\otimes P$ by
$$\alpha=\delta+\Delta.$$
Then $(P,\star,\alpha)$ is an adm-Poisson bialgebra.
\end{pro}

\begin{proof}
Let $x,y\in P$. Set
\begin{align*}
A(x,y):=&-\delta([x,y])+({\rm ad}(x)\otimes{\rm id}+{\rm id}\otimes{\rm ad}(x))\delta(y)-({\rm ad}(y)\otimes{\rm id}+{\rm id}\otimes{\rm ad}(y))\delta(x),\\
B(x,y):=&-\Delta(x\circ y)+({\rm id}\otimes L_\circ(x))\Delta(y)+(L_\circ(y)\otimes{\rm id} )\Delta(x),\\
C(x,y):=&-\Delta([x,y])+({\rm ad}(x)\otimes {\rm id}+{\rm id}\otimes{\rm ad}(x))\Delta(y)+(L_\circ(y)\otimes {\rm id}-{\rm id}\otimes L_\circ(y))\delta(x),\\
D(x,y):=&-\delta(x\circ y)+(L_\circ(y)\otimes {\rm id})\delta(x)+(L_\circ(x)\otimes {\rm id})\delta(y)+({\rm id}\otimes{\rm ad}(x))\Delta(y)\\
        &+({\rm id}\otimes {\rm ad}(y))\Delta(x),\\
E(x,y):=&-\alpha(x\star y)+(R(y)\otimes {\rm id})\alpha(x)+({\rm id}\otimes L(x))\alpha(y)\nonumber\\
        &+\frac{1}{3}\big((L(y)\otimes {\rm id})\alpha(x)-({\rm id}\otimes L(y))\alpha(x)+(L(x)\otimes {\rm id})\alpha(y)-(R(x)\otimes {\rm id})\alpha(y)\nonumber \\
        &+\tau(-\alpha(x\star y)+(L(x)\otimes {\rm id})\alpha(y)+(L(y)\otimes {\rm id})\alpha(x))\big),\\
F(x,y):=&-\alpha(x\star y)+(R(y)\otimes {\rm id})\alpha(x)+({\rm id}\otimes L(x))\alpha(y)\nonumber \\
        &+\frac{1}{3}\big(-\alpha(y\star x)+(L(x)\otimes {\rm id})\alpha(y)+(L(y)\otimes {\rm id})\alpha(x) \nonumber \\
        &+\tau((L(x)\otimes {\rm id})\alpha(y)+(L(y)\otimes {\rm id})\alpha(x)-(R(y)\otimes {\rm id})\alpha(x)-({\rm id}\otimes L(x))\alpha(y))\big),\\
G(x,y):=&(L(y)\otimes {\rm id})\alpha(x)-({\rm id}\otimes R(y))\alpha(x)+\tau\left((L(x)\otimes {\rm id})\alpha(y)-({\rm id}\otimes R(x))\alpha(y)\right)\nonumber \\
        &+\frac{1}{3}\big((R(y)\otimes {\rm id})\alpha(x)-({\rm id}\otimes L(y))\alpha(x)+({\rm id}\otimes L(x))\alpha(y)-(R(x)\otimes {\rm id})\alpha(y)\nonumber \\
        &+\tau(\alpha(y\star x)-\alpha(x\star y))\big).
\end{align*}
Note that $(P,[\;,\;],\circ,\delta,\Delta)$ is a Poisson bialgebra if and only if $(P^*,\delta^*,\Delta^*)$ is a Poisson algebra and
\begin{equation}\label{eq:PB-sPB1}
A(x,y)=B(x,y)=C(x,y)=D(x,y)=0
\end{equation}
and $(P,\c,\alpha)$ is an adm-Poisson bialgebra if and only if
$(P^*,\alpha^*)$ is an adm-Poisson algebra and
\begin{equation}\label{eq:PB-sPB2}
E(x,y)=F(x,y)=G(x,y)=0.
\end{equation}

It is straightforward to show that $(P^*,\delta^*,\Delta^*)$ is a Poisson algebra if and only if $(P^*,\alpha^*)$ is an adm-Poisson algebra.

By a straightforward calculation, we have
\begin{align*}
A(x,y)=&-\frac{1}{4}E(x,y)+\frac{1}{4}E(y,x)+\frac{3}{4}F(x,y)-\frac{3}{4}F(y,x)-\frac{1}{4}G(x,y)+\frac{1}{4}G(y,x);\\
B(x,y)=&\frac{7}{8}E(x,y)+\frac{5}{8}E(y,x)-\frac{3}{8}F(x,y)-\frac{3}{8}F(y,x)-\frac{3}{8}G(x,y)-\frac{1}{4}G(y,x)+\frac{3}{8}\tau G(y,x);\\
C(x,y)=&\frac{1}{4}G(x,y)-\frac{1}{2}G(y,x)+\frac{3}{4}\tau G(x,y);\\
D(x,y)=&-\frac{3}{4}E(x,y)-\frac{3}{4}E(y,x)+\frac{3}{4}F(x,y)+\frac{3}{4}F(y,x)+\frac{1}{4}G(x,y)+\frac{1}{4}G(y,x);\\
E(x,y)=&\frac{2}{3}A(x,y)+\frac{4}{3}B(x,y)+\frac{1}{3}C(x,y)-C(y,x)+D(x,y)+\frac{1}{3}\tau D(x,y);\\
F(x,y)=&\frac{2}{3}A(x,y)+\frac{4}{3}B(x,y)-\frac{2}{3}C(y,x)+\frac{4}{3}D(x,y);\\
G(x,y)=&-\frac{2}{3}A(x,y)+\frac{4}{3}B(x,y)-\frac{4}{3}B(y,x)+\frac{4}{3}C(x,y)+\frac{2}{3}C(y,x).
\end{align*}
Therefore Eq.~(\ref{eq:PB-sPB1}) holds if and only if
Eq.~(\ref{eq:PB-sPB2}) holds.  Hence the conclusion follows.
\end{proof}

\begin{rmk} On the one hand, the above proof generalizes and
illustrates explicitly the polari-zation-depolarization process (\cite{LiLo, MR06,OPV}) in the sense of bialgebras for
Poisson bialgebras with one cocommutative comultiplication $\Delta$ and one anti-cocommutative comultiplication $\delta$ and adm-Poisson bialgebras with one
comultiplication $\alpha$. On the other hand, the above correspondence between adm-Poisson bialgebras and
Poisson bialgebras also can be given equivalently in terms of
Manin triples or matched pairs (\cite{NB1}).
\end{rmk}

\section{A special class of adm-Poisson bialgebras}

In this section, we consider a special class of adm-Poisson
bialgebras, that is, an adm-Poisson bialgebra $(P,\c,\alpha)$ with
$\alpha$ in the form
    \begin{equation}\label{eq:cobo}
    \alpha(x)=({\rm id}\otimes L(x)-R(x)\otimes {\rm id})r,\;\; \forall x\in P.
    \end{equation}
Such adm-Poisson bialgebras are quite similar
as the ``coboundary Lie bialgebras" for Lie algebras
(\cite{CP1,Dr}) or the ``coboundary infinitesimal bialgebras" for
associative algebras (\cite{Aguiar1,Bai2}), although to our knowledge, a well-constructed
cohomology theory for adm-Poisson algebras is not known yet.


\begin{pro}\label{pro:co1}
Let $(P,\c)$ be an adm-Poisson algebra and $r\in P\otimes P$. Let
$\alpha:P\rightarrow P\otimes P$ be a linear map defined by
Eq.~(\ref{eq:cobo}). Then for all $x,y\in P$, we have
\begin{enumerate}
\item $\alpha$ satisfies Eq.~(\ref{Defbi1}) if and only if
\begin{eqnarray}
&&({\rm id}\otimes L(x))(L(y)\otimes {\rm id}-{\rm id}\otimes
R(y))(r+\tau(r))\nonumber\\ \nonumber
  &&+({\rm id}\otimes L(y))(L(x)\otimes {\rm
  id}-{\rm id}\otimes R(x))(r+\tau(r))\\
  &&-(L(x\c y)\otimes {\rm id}-{\rm id}\otimes R(x\c
  y))(r+\tau(r))=0;\label{eq:eqv1}
\end{eqnarray}
\item $\alpha$ satisfies Eq.~(\ref{Defbi2}) if and only if
\begin{eqnarray}
&&({\rm id}\otimes L(x))(L(y)\otimes {\rm id}-{\rm id}\otimes
R(y))(r+\tau(r))\nonumber\\\nonumber
  &&+({\rm id}\otimes L(y))(L(x)\otimes {\rm
  id}-{\rm id}\otimes R(x))(r+\tau(r))\\\nonumber
  &&-(L(x)\otimes {\rm id})(L(y)\otimes {\rm id}-{\rm id}\otimes
  R(y))(r+\tau(r))\\
  &&-({\rm id}\otimes R(y))(L(x)\otimes {\rm id}-{\rm id}\otimes
  R(x))(r+\tau(r))=0;\label{eq:eqv2}
\end{eqnarray}
\item $\alpha$ satisfies Eq.~(\ref{Defbi3}) if and only if
\begin{eqnarray}
&&(R(x)\otimes {\rm id}-{\rm id}\otimes L(x))(L(y)\otimes {\rm
id}-{\rm id}\otimes R(y))(r+\tau(r))\nonumber\\
  &&+\frac{1}{3}\big((L(x\c y)-L(y\c x))\otimes {\rm id}-{\rm id}\otimes(R(x\c y)-R(y\c
  x))\big)(r+\tau(r))=0.\label{eq:eqv3}
\end{eqnarray}
\end{enumerate}
\end{pro}

\begin{proof} We only give an explicit proof for (a) as an example and omit the proof for (b) and (c)  since it is
similar. Let $x,y\in P$. Substituting Eq. $(\ref{eq:cobo})$ into
the Eq. (\ref{Defbi1}), we have
\begin{align*}
0=&-({\rm id}\otimes L(x\c y))r+(R(x\c y)\otimes {\rm id})r+(R(y)\otimes L(x))r-(R(y)R(x)\otimes {\rm id})r\\
&+({\rm id}\otimes L(x)L(y))r-(R(y)\otimes L(x))r\\
&+\frac{1}{3}\big((L(y)\otimes L(x))r-(L(y)R(x)\otimes{\rm id})r-({\rm id}\otimes L(y)L(x))r+(R(x)\otimes L(y))r  \\
&+(L(x)\otimes L(y))r-(L(x)R(y)\otimes {\rm id})r-(R(x)\otimes L(y))r+(R(x)R(y)\otimes {\rm id})r\\
&-(L(x\c y)\otimes {\rm id})\tau(r)+({\rm id}\otimes R(x\c y))\tau(r)+(L(y)\otimes L(x))\tau(r)\\
&-({\rm id}\otimes L(x)R(y))\tau(r)+(L(x)\otimes L(y))\tau(r)-({\rm id}\otimes L(y)R(x))\tau(r)\big)\\
=&\ (A1)+(A2)+(A3)+(A4),
\end{align*}
where
\begin{align*}
(A1)=&\big((-R(y)R(x)+R(x\star y)+\frac{1}{3}(-L(y)R(x)-L(x)R(y)+L(x)L(y)+R(y\star x)))\otimes {\rm id}\big)r,\\
(A2)=&\big({\rm id}\otimes (-L(x\star y)+L(x)L(y)+\frac{1}{3}(-L(y)L(x)+L(y)R(x)+L(x)R(y)-R(x\c y)))\big)r,\\
(A3)=&\frac{1}{3}\big(-L(x)L(y)\otimes {\rm id}-R(y\star x)\otimes {\rm id}+R(x)R(y)\otimes {\rm id}\big)r,\\
(A4)=&\frac{1}{3}\big(-(L(x\c y)\otimes {\rm id})\tau(r)+({\rm id}\otimes R(x\c y))\tau(r)+({\rm id}\otimes R(x\c y))r\\
     &+({\rm id}\otimes L(x))(L(y)\otimes {\rm id}-{\rm id}\otimes R(y))(r+\tau(r))\\
     &+({\rm id}\otimes L(y))(L(x)\otimes {\rm id}-{\rm id}\otimes R(x))(r+\tau(r))\big).
\end{align*}
Since $(L,R,P)$ is a representation of the adm-Poisson algebra
$(P,\c)$, we get $$(A1)=(A2)=0,\;\;(A3)=-\frac{1}{3}(L(x\c
y)\otimes {\rm id})r.$$ Thus we have
\begin{eqnarray*}
(A3)+(A4)&=&\frac{1}{3}\big(({\rm id}\otimes L(x))(L(y)\otimes {\rm id}-{\rm id}\otimes R(y))(r+\tau(r))\\
  &&+({\rm id}\otimes L(y))(L(x)\otimes {\rm id}-{\rm id}\otimes R(x))(r+\tau(r))\\
  &&-(L(x\c y)\otimes {\rm id}-{\rm id}\otimes R(x\c y))(r+\tau(r))\big)=0.
\end{eqnarray*}
Therefore Eq. $(\ref{Defbi1})$ holds if and only if Eq.~(\ref{eq:eqv1})
holds.
\delete
{(b) Substituting Eq. $(\ref{eq:cobo})$ into the
Eq. (\ref{Defbi2}), we have
\begin{align*}
0=&-(R(x\c y)\otimes {\rm id})r+({\rm id}\otimes L(x\c y))r+(R(y)R(x)\otimes {\rm id})r-(R(y)\otimes L(x))r\\
&+(R(y)\otimes L(x))r-({\rm id}\otimes L(x)L(y))r\\
&+\frac{1}{3}\big((L(y)R(x)\otimes{\rm id})r-(L(y)\otimes L(x))r-(R(y\c x)\otimes {\rm id})r+({\rm id}\otimes L(y\c x))r  \\
&+(L(x)R(y)\otimes {\rm id})r-(L(x)\otimes L(y))r+(L(x)L(y)\otimes {\rm id})\tau(r)-(L(x)\otimes R(y))\tau(r)\\
&-({\rm id}\otimes R(y)R(x))\tau(r)+(L(x)\otimes R(y))\tau(r)+({\rm id}\otimes L(x)R(y))\tau(r)\\
&-(L(y)\otimes L(x))\tau(r)+({\rm id}\otimes L(y)R(x))\tau(r)-(L(x)\otimes L(y))\tau(r)\big)\\
=&\ (F1)+(F2)+(F3)+(F4),
\end{align*}
where
\begin{align*}
(F1)=&\big((R(y)R(x)-R(x\star y)+\frac{1}{3}(L(y)R(x)+L(x)R(y)-L(x)L(y)-R(y\star x)))\otimes {\rm id}\big)r,\\
(F2)=&\big({\rm id}\otimes (L(x\star y)-L(x)L(y)+\frac{1}{3}(L(y)L(x)-L(y)R(x)-L(x)R(y)+R(x\c y)))\big)r,\\
(F3)=&\frac{1}{3}\big(-{\rm id}\otimes L(y)L(x)-{\rm id}\otimes R(x\star y)+{\rm id}\otimes L(y\c x)\big)r,\\
(F4)=&\frac{1}{3}\big(-({\rm id}\otimes R(y)R(x))\tau(r)+(L(x)\otimes R(y))(r+\tau(r))\\
     &+({\rm id}\otimes L(x))({\rm id}\otimes R(y)-L(y)\otimes {\rm id})(r+\tau(r))\\
     &+({\rm id}\otimes L(y))({\rm id}\otimes R(x)-L(x)\otimes {\rm id})(r+\tau(r))\\
     &+(L(x)\otimes{\rm id})(L(y)\otimes{\rm id}-{\rm id}\otimes R(y))(r+\tau(r))\big).
\end{align*}
Since $(L,R,P)$ is a representation of the adm-Poisson algebra
$(P,\c)$, we get $$(F1)=(F2)=0,\quad (F3)=-\frac{1}{3}({\rm
id}\otimes R(y)R(x))r.$$ Thus, one has
\begin{align*}
(F3)+(F4)=&\frac{1}{3}\big(({\rm id}\otimes R(y))(L(x)\otimes {\rm id}-{\rm id}\otimes R(x))(r+\tau(r))\\
  &+({\rm id}\otimes L(y))({\rm id}\otimes R(x)-L(x)\otimes {\rm id})(r+\tau(r))\\
  &+(L(x\c y)\otimes {\rm id}-{\rm id}\otimes R(x\c y)(r+\tau(r))\\
  &+(L(x)\otimes{\rm id})(L(y)\otimes{\rm id}-{\rm id}\otimes R(y))(r+\tau(r))\big)=0.
\end{align*}
This implies that Eq. $(\ref{Defbi2})$ holds if and only if Eq.~(\ref{eq:eqv2})
holds.

(c) Substituting Eq. $(\ref{eq:cobo})$ into
Eq. (\ref{Defbi3}), we have
\begin{align*}
0=&-(R(x\c y)\otimes {\rm id})r+({\rm id}\otimes L(x\c y))r+(R(y)R(x)\otimes {\rm id})r-(R(y)\otimes L(x))r\\
&+(R(y)\otimes L(x))r-({\rm id}\otimes L(x)L(y))r\\
&+\frac{1}{3}\big((L(y)R(x)\otimes{\rm id})r-(L(y)\otimes L(x))r-(R(y\c x)\otimes {\rm id})r+({\rm id}\otimes L(y\c x))r  \\
&+(L(x)R(y)\otimes {\rm id})r-(L(x)\otimes L(y))r+(L(x)L(y)\otimes {\rm id})\tau(r)-(L(x)\otimes R(y))\tau(r)\\
&-({\rm id}\otimes R(y)R(x))\tau(r)+(L(x)\otimes R(y))\tau(r)+({\rm id}\otimes L(x)R(y))\tau(r)\\
&-(L(y)\otimes L(x))\tau(r)+({\rm id}\otimes L(y)R(x))\tau(r)-(L(x)\otimes L(y))\tau(r)\big)\\
=&\ (G1)+(G2)+(G3)+(G4)+(G5),
\end{align*}
where
\begin{align*}
(G1)=&\big((L(y)R(x)-R(y)L(x)+\frac{1}{3}(L(y)L(x)-L(x)L(y)-R(y\star x)+R(x\star y)))\otimes {\rm id}\big)r,\\
(G2)=&\big({\rm id}\otimes(R(y)L(x)-L(x)R(y)+\frac{1}{3}(L(y)L(x)-L(x)L(y)-R(y\c x)+R(x\c y)))\big)r,\\
(G3)=&\frac{1}{3}\big((-R(x)R(y)-L(x\c y)+L(x)L(y)+R(y\star x))\otimes {\rm id}\big)r,\\
(G4)=&\frac{1}{3}\big((R(y)R(x)+L(y\c x)-L(y)L(x)-R(x\star y))\otimes {\rm id}\big)r,\\
(G5)=&(R(x)\otimes {\rm id}-{\rm id}\otimes L(x))(L(y)\otimes {\rm id}-{\rm id}\otimes R(y))(r+\tau(r))\\
     &+\frac{1}{3}\big((L(x\c y)-L(y\c x))\otimes {\rm id}-{\rm id}\otimes(R(x\c y)-R(y\c x))\big)(r+\tau(r)).
\end{align*}
Since $(L,R,P)$ is a representation of the adm-Poisson algebra, we
have
$$(G1)=(G2)=0.$$
Furthermore, by Eq. \eqref{lem:rep1}, we have
$$L(x\c y)+R(x)R(y)=L(x)L(y)+R(y\c x),\quad \forall x,y\in P,$$
which implies that
$$(G3)=(G4)=0.$$
Therefore, Eq. $(\ref{Defbi3})$ holds if and only if Eq.~(\ref{eq:eqv3})
holds.}
\end{proof}

\begin{pro}\label{pro:cosp}
Let $(P,\c)$ be an adm-Poisson algebra and $r=\sum a_i\otimes
b_i\in P\otimes P$. Let $\alpha:P\rightarrow P\otimes P$ be a
linear map defined by Eq.~(\ref{eq:cobo}). Then 
$\alpha$ satisfies Eq.~(\ref{eq:coalgebra}) if and only if $r$
satisfies
\begin{eqnarray}\label{eq:cosp}
&&(R(x)\otimes {\rm id}\otimes{\rm id}-{\rm id}\otimes{\rm
id}\otimes L(x)){\bf P}(r)+\frac{1}{3}\Big(({\rm id}\otimes R(x)\otimes {\rm
id}-{\rm id}\otimes{\rm id}\otimes
R(x)){\bf P}(r)\nonumber\\
  &&+\left(R(x)\otimes{\rm id}\otimes{\rm id}-{\rm id}\otimes R(x)\otimes{\rm id}\right){\bf Q}(r)\nonumber\\
  &&-\sum_i(R(x)\otimes {\rm id}\otimes {\rm id})\big(a_i\otimes((L(b_i)\otimes{\rm id}-{\rm id}\otimes R(b_i))(r+\tau(r)))\big)\nonumber\\
  &&-\sum_i(({\rm id}\otimes R(x)\otimes{\rm id})(\tau \otimes{\rm id}))\big(a_i\otimes((L(b_i)\otimes{\rm id}-{\rm id}\otimes R(b_i))(r+\tau(r)))\big)\nonumber\\
  &&+\sum_i({\rm id}\otimes \tau+{\rm id})\big((L(a_i)\otimes {\rm id})(L(x)\otimes {\rm id}-{\rm id}\otimes R(x))(r+\tau(r))\otimes b_i\big)\nonumber\\
  &&+\sum_i(\tau \otimes {\rm id}+{\rm id})\big(a_i\otimes((L(x\c b_i)\otimes {\rm id}-{\rm id}\otimes R(x\c b_i))(r+\tau(r)))\big)\nonumber\\
  &&-\sum_ia_i\otimes \big((R(b_i)\otimes {\rm id})(L(x)\otimes {\rm id}-{\rm id}\otimes R(x))(r+\tau(r))\big)\nonumber\\
  &&-\sum_i\big((R(a_i)\otimes {\rm id})(L(x)\otimes {\rm id}-{\rm id}\otimes R(x))(r+\tau(r))\big)\otimes
  b_i\Big)=0,
\end{eqnarray}
for all $x\in P$, where
\begin{eqnarray}
{\bf P}(r)=r_{23}\c r_{12}-r_{13}\c r_{23}-r_{12}\c r_{13},
\;\;{\bf Q}(r)=r_{12}\c r_{23}-r_{23}\c r_{13}-r_{13}\c r_{12}.
\end{eqnarray}

\end{pro}

\begin{proof}
Let $x\in A$. Then we have
\begin{eqnarray}
\nonumber && ({\rm id}\otimes\alpha)\alpha(x)-(\alpha\otimes{\rm
id})\alpha(x)-(R(x)\otimes {\rm id}\otimes{\rm id}-{\rm id}\otimes{\rm
id}\otimes L(x)){\bf P}(r)\\
\nonumber&=&\sum_{i,j}\big(a_i\otimes a_j\otimes \big((x\star b_i)\star b_j-x\star (b_i\star b_j)\big)+a_i\otimes\big((a_j\star x)\star b_i-a_j\star (x\star b_i)\big)\otimes b_j\\
\nonumber&&+\big((a_j\star a_i)\star x-a_j\star (a_i\star x)\big)\otimes b_j\otimes b_i\big)\\
\nonumber&=&\frac{1}{3}\sum_{i,j}\Big(a_i\otimes a_j\otimes \big(x\star (b_j\star b_i)- b_j\star (x\star b_i) - b_i\star (x\star b_j)+ b_i\star (b_j\star x)\big)\\
\nonumber&&+a_i\otimes\big(a_j\star (b_i \star x)-b_i\star (a_j \star x)-x\star (a_j \star b_i)+x\star (b_i \star a_j)\big)\otimes b_j \\
&&+\big(a_j\star (x\star a_i)-x\star (a_j\star a_i)-a_i\star
(a_j\star x)+a_i\star (x\star a_j)\big)\otimes b_j\otimes b_i
\Big).\label{eq:mainthm1}
\end{eqnarray}
Furthermore, by Eq.~\eqref{eq:adm-Poisson property1}, we have
\begin{eqnarray}
  \nonumber&&a_i\otimes a_j\otimes \big(x\star (b_j\star b_i)+ b_i\star (b_j\star x)\big)
  = a_i\otimes a_j\otimes \big((x\star b_j)\star b_i+ (b_i\star b_j)\star x\big)\\
\label{eq:mainthm2}  &&=a_i\otimes a_j\otimes(x\star b_j)\star b_i+(\id\otimes\id\otimes R(x))(r_{13}\star r_{12}).
\end{eqnarray}
Similarly, we have
{\small \begin{eqnarray}
&&a_i\otimes\big(a_j\star (b_i \star x)+x\star (b_i \star a_j)\big)\otimes b_j=a_i\otimes (x\star b_i) \star a_j \otimes b_j+(\id\otimes R(x)\otimes\id)(r_{23}\star r_{12});\label{eq:mainthm3}\\
&&a_i\otimes\big(b_i\star (a_j \star x)+x\star (a_j \star b_i)\big)\otimes b_j=a_i\otimes (x\star a_j)\star b_i \otimes b_j+(\id\otimes R(x)\otimes\id)(r_{12}\star r_{23});\label{eq:mainthm4}\\
 &&\big(x\star (a_j\star a_i)+a_i\star (a_j\star x)\big)\otimes b_j\otimes b_i=(x\star a_j)\star a_i\otimes b_j\otimes b_i+(R_x\otimes\id\otimes\id)(r_{13}\star r_{12}).\label{eq:mainthm5}
\end{eqnarray}}
Substituting Eqs. \eqref{eq:mainthm2}-\eqref{eq:mainthm5} into Eq. \eqref{eq:mainthm1}, by a direct calculation, we have
\begin{eqnarray*}
  \nonumber && ({\rm id}\otimes\alpha)\alpha(x)-(\alpha\otimes{\rm
id})\alpha(x)-(R(x)\otimes {\rm id}\otimes{\rm id}-{\rm id}\otimes{\rm
id}\otimes L(x)){\bf P}(r)\\
&=&\frac{1}{3}\big(({\rm id}\otimes R(x)\otimes {\rm
id}-{\rm id}\otimes{\rm id}\otimes
R(x)){\bf P}(r)+(R(x)\otimes{\rm id}\otimes{\rm id}-{\rm id}\otimes R(x)\otimes{\rm id}){\bf Q}(r)\big)\\
&&+\frac{1}{3}\sum_{i,j}\big( a_i\otimes a_j\otimes(x\star b_j)\star b_i-a_i\otimes a_j\otimes b_j\star (x\star b_i)-a_i\otimes a_j\otimes b_i\star (x\star b_j)\\
&&+a_i\otimes (x\star b_i) \star a_j \otimes b_j-a_i\otimes (x\star a_j)\star b_i \otimes b_j-(x\star a_j)\star a_i\otimes b_j\otimes b_i\\
&&+a_j \star(x\star a_i)\otimes b_j\otimes b_i+a_i \star(x\star a_j)\otimes b_j\otimes b_i-(\id\otimes R(x)\otimes\id)(r_{23}\star r_{13})\\
&&-(\id\otimes R(x)\otimes\id)(r_{13}\star r_{12})-(\id\otimes \id\otimes R(x))(r_{12}\star r_{23})+(\id\otimes R(x)\otimes\id)(r_{12}\star r_{13})\\
&&+(\id\otimes R(x)\otimes\id)(r_{13}\star r_{23})+(\id\otimes \id\otimes R(x))(r_{23}\star r_{12})+(R_x\otimes\id\otimes\id)(r_{23}\star r_{13})\\
&&-(R_x\otimes\id\otimes\id)(r_{12}\star r_{23})\big).
\end{eqnarray*}
On the other hand, we have
{\small \begin{eqnarray*}
&& ({\rm id}\otimes \tau)({\rm
id}\otimes \alpha)\alpha(x)-(\tau\otimes{\rm id})({\rm
id}\otimes\alpha)\alpha(x)-({\rm id}\otimes\tau)(\tau\otimes{\rm id})({\rm id}\otimes\alpha)\alpha(x)\\
&&+(\tau\otimes{\rm id})({\rm id}\otimes\tau)({\rm
      id}\otimes\alpha)\alpha(x)\\
&=&\sum_{i,j}\Big(a_i\otimes (x\star b_i)\star b_j\otimes a_j-a_i\otimes b_j\otimes a_j\star(x\star b_i)-a_i\star x\otimes b_i\star b_j\otimes a_j+a_i\star x\otimes b_j\otimes a_j\star b_i\\
&&-a_i\otimes a_j\otimes (x\star b_j)\star b_i+a_j\star(x\star b_i) \otimes a_i\otimes b_j+a_j\otimes a_i\star x\otimes b_i\star b_j
-a_j\star b_i\otimes a_i\star x\otimes b_j\\&&-a_j\otimes (x\star b_i)\star b_j\otimes a_i+a_j\star(x\star b_i) \otimes b_j\otimes a_i+a_j\otimes b_i\star b_j\otimes a_i\star x
-a_j\star b_i\otimes b_j\otimes a_i\star x\\&&+(x\star b_i)\star b_j \otimes a_i\otimes a_j-b_j\otimes a_i\otimes a_j\star(x\star b_i)-b_i\star b_j\otimes a_i\star x\otimes a_j+b_j\otimes a_i\star x\otimes a_j\star b_i\Big).
\end{eqnarray*}}
Thus we have
\begin{eqnarray*}
  &&({\rm id}\otimes\alpha)\alpha(x)-(\alpha\otimes{\rm
id})\alpha(x)+\frac{1}{3}\big(({\rm id}\otimes \tau)({\rm
id}\otimes \alpha)\alpha(x)-(\tau\otimes{\rm id})({\rm
id}\otimes\alpha)\alpha(x)\\\nonumber
      &&-({\rm id}\otimes\tau)(\tau\otimes{\rm id})({\rm id}\otimes\alpha)\alpha(x)+(\tau\otimes{\rm id})({\rm id}\otimes\tau)({\rm
      id}\otimes\alpha)\alpha(x)\big)\\
&=&(R(x)\otimes {\rm id}\otimes{\rm id}-{\rm id}\otimes{\rm
id}\otimes L(x)){\bf P}(r)+\frac{1}{3}\Big(({\rm id}\otimes R(x)\otimes {\rm
id}-{\rm id}\otimes{\rm id}\otimes
R(x)){\bf P}(r)\nonumber\\
  &&+\left(R(x)\otimes{\rm id}\otimes{\rm id}-{\rm id}\otimes R(x)\otimes{\rm id}\right){\bf Q}(r)\Big)
  +\frac{1}{3} (-A-B+C+D-E-F),
\end{eqnarray*}
where
{\small\begin{eqnarray*}
 &&A= \sum_{i,j}\big(a_i\star x\otimes b_i\star a_j\otimes b_j+a_i\star x\otimes b_i\star b_j\otimes a_j -a_i\star x\otimes a_j\otimes b_j\star b_i-a_i\star x\otimes b_j\otimes a_j\star b_i\big);\\
 &&B= \sum_{i,j}\big(b_i\star a_j\otimes a_i\star x\otimes b_j
  +b_i\star b_j\otimes a_i\star x\otimes a_j -a_j\otimes a_i\star x\otimes b_j\star b_i-b_j\otimes a_i\star x\otimes a_j\star b_i\big);\\
&& C=\sum_{i,j}\big(a_i\star (x\star a_j)\otimes b_j\otimes b_i+a_i\star(x\star b_j)\otimes a_j\otimes b_i
  -a_i\star a_j\otimes b_j\star x\otimes b_i-a_i\star b_j\otimes a_j\star x \otimes b_i\\
  &&\mbox{}\hspace{1cm} +a_i\star (x\star a_j)\otimes b_i\otimes b_j+a_i\star (x\star b_j)\otimes b_i\otimes a_j-a_i\star a_j\otimes b_i\otimes b_j\star x-a_i\star b_j\otimes b_i\otimes a_j\star x\big);\\
&&D=\sum_{i,j}\big(a_i\otimes (x\star b_i)\star a_j\otimes b_j+a_i\otimes (x\star b_i)\star b_j\otimes a_j-a_i\otimes a_j \otimes b_j\star(x\star b_i)
  -a_i\otimes b_j \otimes a_j\star(x\star b_i)\\
  &&\mbox{}\hspace{1cm}+(x\star b_i)\star a_j\otimes a_i\otimes b_j+(x\star b_i)\star b_j\otimes a_i\otimes a_j
  -a_j\otimes a_i \otimes b_j\star(x\star b_i)-b_j\otimes a_i \otimes a_j\star(x\star b_i)\big);\\
&& E=\sum_{i,j}\big(a_i\otimes (x\star a_i)\star b_i\otimes b_j+a_i\otimes (x\star b_j)\star b_i\otimes a_j-a_i\otimes a_j\star b_i\otimes b_j\star x-a_i\otimes b_j\star b_i\otimes a_j\star x\big);\\
&&F=\sum_{i,j}\big( (x\star a_j)\star a_i\otimes b_j\otimes b_i+(x\star b_j)\star a_i\otimes a_j\otimes b_i-a_j\star a_i\otimes b_j\star x\otimes b_i-b_j\star a_i\otimes a_j\star x\otimes b_i\big).
\end{eqnarray*}}
Furthermore, it is straightforward to check that
\begin{eqnarray*}
 A&=&\sum_i(R(x)\otimes {\rm id}\otimes {\rm id})\big(a_i\otimes((L(b_i)\otimes{\rm id}-{\rm id}\otimes R(b_i))(r+\tau(r)))\big);\\
B&=&\sum_i(({\rm id}\otimes R(x)\otimes{\rm id})(\tau \otimes{\rm id}))\big(a_i\otimes((L(b_i)\otimes{\rm id}-{\rm id}\otimes R(b_i))(r+\tau(r)))\big);\\
C&=&\sum_i({\rm id}\otimes \tau+{\rm id})\big((L(a_i)\otimes {\rm id})(L(x)\otimes {\rm id}-{\rm id}\otimes R(x))(r+\tau(r))\otimes b_i\big);\\
D&=&\sum_i(\tau \otimes {\rm id}+{\rm id})\big(a_i\otimes((L(x\c b_i)\otimes {\rm id}-{\rm id}\otimes R(x\c b_i))(r+\tau(r)))\big);\\
E&=&\sum_ia_i\otimes \big((R(b_i)\otimes {\rm id})(L(x)\otimes {\rm id}-{\rm id}\otimes R(x))(r+\tau(r))\big);\\
F&=&\sum_i\big((R(a_i)\otimes {\rm id})(L(x)\otimes {\rm id}-{\rm
id}\otimes R(x))(r+\tau(r))\big)\otimes
  b_i.
\end{eqnarray*}
By Lemma~\ref{lem:cosp}, the conclusion follows.
\end{proof}

Combining Propositions~\ref{pro:co1} and \ref{pro:cosp}, we have
the following result.

\begin{thm}\label{co2}
Let $(P,\c)$ be an adm-Poisson algebra and $r\in P\otimes P$. Let
$\alpha:P\rightarrow P\otimes P$ be a linear map defined by
Eq.~(\ref{eq:cobo}). Then $(P,\c,\alpha)$ is an adm-Poisson
bialgebra if and only if $r$ satisfies
Eqs.~(\ref{eq:eqv1})-(\ref{eq:cosp}).
\end{thm}

\begin{cor}\label{pro:coboundary adm-Poisson} Let $(P,\c)$ be
an adm-Poisson algebra and $r\in P\otimes P$. Let
$\alpha:P\rightarrow P\otimes P$ be a linear map defined by
Eq.~(\ref{eq:cobo}). If $r$ satisfies
\begin{equation}\label{eq:con1}
(L(x)\otimes {\rm id} -{\rm id}\otimes R(x))(r+\tau(r))=0,\quad\forall x\in P,
\end{equation}
 then $(P,\c,\alpha)$ is an adm-Poisson
bialgebra if and only if $r$ satisfies
\begin{eqnarray}\label{eq:cosp2}
&&(R(x)\otimes {\rm id}\otimes{\rm id}-{\rm id}\otimes{\rm
id}\otimes L(x)){\bf P}(r)+\frac{1}{3}\big(({\rm id}\otimes R(x)\otimes {\rm
id}-{\rm id}\otimes{\rm id}\otimes
R(x)){\bf P}(r)\nonumber\\
  &&+\left(R(x)\otimes{\rm id}\otimes{\rm id}-{\rm id}\otimes R(x)\otimes{\rm id}\right){\bf Q}(r)\big)=0,\;\;\forall x\in P.
\end{eqnarray}
\end{cor}

\begin{proof}
  Since $r$ satisfies Eq. \eqref{eq:con1}, Eqs.~(\ref{eq:eqv1})-(\ref{eq:eqv3}) hold naturally and Eq. \eqref{eq:cosp} is equivalent to Eq. \eqref{eq:cosp2}. The conclusion follows immediately.
\end{proof}

\emptycomment{\begin{thm} Let $(P,\c,\alpha_P)$ be an adm-Poisson
bialgebra. Then there is an adm-Poisson bialgebra structure on the
direct sum $P\oplus P^*$ of the underlying vector spaces of $P$
and $P^*$.
\end{thm}

\begin{proof}
It can be obtained from \cite[Theorem 3]{NB1} from the
correspondence between adm-Poisson bialgebras and Poisson
bialgebras or a direct proof as follows. We only give the outline
here. Let $\{e_1,\cdots, e_n\}$ be a basis of $P$ and
$\{e_1*,\cdots,e_n^*\}$ be the dual basis. Let $r=\sum_i
e_i\otimes e_i^* \in P\otimes P^*\subset (P\oplus
P^*)\otimes(P\oplus P^*)$. It is straightforward to check that $r$
satisfies Eqs.~(\ref{eq:eqv1})-(\ref{eq:eqv3}) and (\ref{eq:cosp})
on the adm-Poisson algebra
$P\bowtie^{-R_P^*,-L_P^*}_{-R^*_{P^*},-L_{P^*}^*}P^*$. Thus,
\begin{equation*}
    \alpha(x)=(R(x)\otimes{\rm id}-{\rm id}\otimes L(x))r,\quad \forall x\in P\oplus P^*,
\end{equation*}
defines a  adm-Poisson bialgebra structure on $P\oplus P^*$.
\end{proof}}

\begin{lem}\label{lem:eqiv}
Let $(P,\c)$ be an adm-Poisson algebra. If $r\in P\otimes P$ satisfies Eq.~\eqref{eq:con1}, then
$${\bf P}(r)=r_{23}\c r_{12}-r_{13}\c r_{23}-r_{12}\c r_{13}=0$$
if and only if
$${\bf Q}(r)=r_{12}\c r_{23}-r_{23}\c r_{13}-r_{13}\c r_{12}=0.$$
\end{lem}

\begin{proof} Let $\frka$ and $\frks$ be the skew-symmetric and symmetric parts of $r$ respectively, that is,
$$\frka=\frac{1}{2}(r-\sigma(r)),\;\;\frks=\frac{1}{2}(r+\sigma(r)).$$
Then $r=\frka+\frks$. Let $a^*,b^*,c^*\in P^*$. Then we have
 {\small \begin{eqnarray}
 \nonumber&&( r_{23}\c r_{12}-r_{13}\c r_{23}-r_{12}\c r_{13})(a^*,b^*,c^*)\\
 \nonumber&&=-\langle\frka^\sharp(c^*)\star \frka^\sharp(a^*),b^*\rangle-\langle\frka^\sharp(a^*)\star \frka^\sharp(b^*),c^*\rangle-\langle\frka^\sharp(b^*)\star \frka^\sharp(c^*),a^*\rangle+\langle\frks^\sharp(c^*)\star \frks^\sharp(a^*),b^*\rangle\\
 \nonumber&&\mbox{}\hspace{0.25cm}-\langle\frks^\sharp(a^*)\star \frks^\sharp(b^*),c^*\rangle-\langle\frks^\sharp(b^*)\star \frks^\sharp(c^*),a^*\rangle-\langle\frka^\sharp(c^*)\star \frks^\sharp(a^*),b^*\rangle+\langle\frks^\sharp(c^*)\star \frka^\sharp(a^*),b^*\rangle\\
 \label{eq:adm-PYBE1}&&\mbox{}\hspace{0.25cm}-\langle\frka^\sharp(a^*)\star \frks^\sharp(b^*),c^*\rangle-\langle\frks^\sharp(a^*)\star \frka^\sharp(b^*),c^*\rangle+\langle\frka^\sharp(b^*)\star \frks^\sharp(c^*),a^*\rangle+\langle\frks^\sharp(b^*)\star \frka^\sharp(c^*),a^*\rangle;\\
\nonumber&&( r_{12}\c r_{23}-r_{23}\c r_{13}-r_{13}\c r_{12})(c^*,b^*,a^*)\\
 \nonumber&&=-\langle\frka^\sharp(c^*)\star \frka^\sharp(a^*),b^*\rangle-\langle\frka^\sharp(a^*)\star \frka^\sharp(b^*),c^*\rangle-\langle\frka^\sharp(b^*)\star \frka^\sharp(c^*),a^*\rangle+\langle\frks^\sharp(c^*)\star \frks^\sharp(a^*),b^*\rangle\\
 \nonumber&&-\langle\frks^\sharp(a^*)\star \frks^\sharp(b^*),c^*\rangle-\langle\frks^\sharp(b^*)\star \frks^\sharp(c^*),a^*\rangle+\langle\frka^\sharp(c^*)\star \frks^\sharp(a^*),b^*\rangle-\langle\frks^\sharp(c^*)\star \frka^\sharp(a^*),b^*\rangle\\
 \label{eq:adm-PYBE2}&&+\langle\frka^\sharp(a^*)\star \frks^\sharp(b^*),c^*\rangle+\langle\frks^\sharp(a^*)\star \frka^\sharp(b^*),c^*\rangle-\langle\frka^\sharp(b^*)\star \frks^\sharp(c^*),a^*\rangle-\langle\frks^\sharp(b^*)\star \frka^\sharp(c^*),a^*\rangle.
 \end{eqnarray}}
 Thus we have
 \begin{eqnarray*}
  && ( r_{23}\c r_{12}-r_{13}\c r_{23}-r_{12}\c r_{13})(a^*,b^*,c^*)-( r_{12}\c r_{23}-r_{23}\c r_{13}-r_{13}\c r_{12})(c^*,b^*,a^*)\\
  &=&2\big(-\langle\frka^\sharp(c^*)\star \frks^\sharp(a^*),b^*\rangle+\langle\frks^\sharp(c^*)\star \frka^\sharp(a^*),b^*\rangle-\langle\frka^\sharp(a^*)\star \frks^\sharp(b^*),c^*\rangle\\
 &&-\langle\frks^\sharp(a^*)\star \frka^\sharp(b^*),c^*\rangle+\langle\frka^\sharp(b^*)\star \frks^\sharp(c^*),a^*\rangle+\langle\frks^\sharp(b^*)\star \frka^\sharp(c^*),a^*\rangle\big).
 \end{eqnarray*}
 Furthermore, by Eq.~\eqref{eq:con1}, we have
 $$\langle\frks^\sharp({L}^*(x)a^*),b^*\rangle=\langle a^*,\frks^\sharp({R}^*(x)b^*)\rangle,\quad\forall~x\in P,a^*,b^*\in P^*.$$
 Then we have
 \begin{eqnarray*}
   \langle\frka^\sharp(c^*)\star \frks^\sharp(a^*),b^*\rangle=-\langle \frks^\sharp(a^*), L^*(\frka^\sharp(c^*))b^*\rangle
   =-\langle \frks^\sharp(R^*(\frka^\sharp(c^*))a^*), b^*\rangle=\langle\frks^\sharp(b^*)\star \frka^\sharp(c^*),a^*\rangle.
 \end{eqnarray*}
 Similarly, we have
 \begin{eqnarray*}
  \langle\frks^\sharp(c^*)\star \frka^\sharp(a^*),b^*\rangle=\langle\frka^\sharp(a^*)\star \frks^\sharp(b^*),c^*\rangle;
  \langle\frka^\sharp(b^*)\star \frks^\sharp(c^*),a^*\rangle=\langle\frks^\sharp(a^*)\star \frka^\sharp(b^*),c^*\rangle.
 \end{eqnarray*}
 Thus
 $$( r_{23}\c r_{12}-r_{13}\c r_{23}-r_{12}\c r_{13})(a^*,b^*,c^*)=( r_{12}\c r_{23}-r_{23}\c r_{13}-r_{13}\c r_{12})(c^*,b^*,a^*).$$
 The conclusion follows immediately.
\end{proof}

\begin{cor}
Let $(P,\c)$ be an adm-Poisson algebra and $r\in P\otimes P$.
Let $\alpha:P\rightarrow
P\otimes P$ be a linear map defined by Eq.~(\ref{eq:cobo}). Then
$(P,\c,\alpha)$ is an adm-Poisson bialgebra if  $r$ satisfies Eq.~(\ref{eq:con1}) and the following equation
\begin{equation}\label{eq1}
{\bf P}(r)=r_{23}\c r_{12}-r_{13}\c r_{23}-r_{12}\c r_{13}=0.
\end{equation}
In particular, if $r$ is skew-symmetric and $r$ satisfies Eq.~(\ref{eq1}), then $(P,\c,\alpha)$ is an adm-Poisson bialgebra.
\end{cor}

\begin{proof}
By Lemma~\ref{lem:eqiv}, it follows as a direct consequence of
Corollary~\ref{pro:coboundary adm-Poisson}.
\end{proof}

\begin{defi}
Let $(P,\c)$ be an adm-Poisson algebra and $r\in P\otimes P$. Eq.~
$(\ref{eq1})$ is called the \textbf{\textup{adm-Poisson
Yang-Baxter equation (adm-PYBE)}} in $(P,\c)$.
\end{defi}

\begin{rmk}
The notion of adm-Poisson  Yang-Baxter equation in an adm-Poisson
algebra is due to the fact that it is an analogue of the classical
Yang-Baxter equation in a Lie algebra (\cite{CP1,Dr}) or the
associative Yang-Baxter equation in an associative algebra
(\cite{Aguiar1,Bai2}).
\end{rmk}

It is an unexpected consequence that both the adm-Poisson
Yang-Baxter equation in an adm-Poisson algebra and the associative
Yang-Baxter equation (\cite{Aguiar1,Bai2}) in an associative
algebra have the same form (\ref{eq1}) (also see
Eq.~(\ref{eq:AYBE})). Hence both these two equations have some
common properties. Next we give two properties of the adm-Poisson
Yang-Baxter equation whose proofs are omitted since the proofs are
the same as in the case of the associative Yang-Baxter equation.

\emptycomment{Let $P$ be a vector space. For all $r\in P\otimes
P$, $r$ can be regarded as a linear map from $P^*$ to $P$ in the
following way:
\begin{equation}\langle  u^*\otimes v^*,r\rangle   =\langle  u^*,r(v^*)\rangle   ,\;\;\forall\; u^*,v^*\in P^*.\end{equation}}

\begin{pro}\label{pro:of}
    Let $(P,\c)$ be an adm-Poisson algebra and $r\in P\otimes P$ be skew-symmetric.
Then $r$ is solution of adm-Poisson Yang-Baxter equation if and
only
    if $r$ satisfies
    \begin{equation}\label{eq:of}
    r^\sharp(a^*)\c r^\sharp(b^*)=r(-R^*(r^\sharp(a^*))b^*-L^*(r^\sharp(b^*))a^*),\;\;\forall a^*,b^*\in P^*.
    \end{equation}
\end{pro}

\begin{rmk}
Since the dual representations of both adm-Poisson algebras and
associative algebras have the same form (see
Remark~\ref{rmk:same}), the interpretation of adm-Poisson
Yang-Baxter equation in terms of the operator form ~(\ref{eq:of}) in
the above Proposition~\ref{pro:of} explains partly why the
adm-Poisson Yang-Baxter equation has the same form as of the
associative Yang-Baxter equation.
\end{rmk}

\begin{thm}  Let $(P, \c)$ be an adm-Poisson algebra and $r\in P\otimes P$.
 Suppose that $r$ is skew-symmetric and
nondegenerate. Then $r$ is a solution of adm-Poisson Yang-Baxter
equation in $(P,\c)$ if and only if the inverse of the isomorphism
$P^*\rightarrow P$ induced by $r$, regarded as a bilinear form
$\omega$ on $P$ (that is, $\omega(x,y)=\langle
(r^\sharp)^{-1}(x),y\rangle$ for all $x,y\in P$), satisfies
\begin{equation}\label{eq:Conn}
    \omega(x\c y, z)+\omega(y\c z, x)+\omega(z\c x, y)=0,\;\;\forall x,y,z\in P.
    \end{equation}
\end{thm}

At the end of this section, we study the relationship between the
adm-Poisson bialgebras $(P,\c,\alpha)$ with $\alpha$ defined by
Eq.~(\ref{eq:cobo}) and the coboundary Poisson bialgebras given in
Section 2.

Let $(P,[\;,\;],\circ,\delta,\Delta)$  be a coboundary Poisson
bialgebra with  $\delta, \Delta$ defined by Eqs. $(\ref{eq:PB1})$
and $(\ref{eq:PB2})$ through $r\in P\otimes P$ respectively.
Define a linear map $\alpha:P\to P\otimes P$ by
\begin{equation}\label{eq:alpha}\alpha(x)=\delta(x)+\Delta(x),\quad \forall x\in
P.\end{equation} Then we have
\begin{equation*}
\alpha(x)=\delta(x)+\Delta(x)=({\rm id}\otimes L(x)-R(x)\otimes{\rm id})r,\quad \forall x\in P.
\end{equation*}
Therefore, by Proposition \ref{cbialgequ}, every coboundary
Poisson bialgebra naturally induces an adm-Poisson bialgebra
structure $(P,\c,\alpha)$ with $\alpha$ satisfying
Eq.~(\ref{eq:cobo}).

Conversely, let $(P,\c,\alpha)$ be an adm-Poisson bialgebra
with $\alpha$ defined by Eq.~(\ref{eq:cobo}) through $r\in
P\otimes P$. Define two linear maps $\delta,\Delta:P\rightarrow
P\otimes P$ as follows:
\begin{eqnarray}\label{eq:Deltadelta}
 \delta(x)=\frac{1}{2}(\alpha(x)-\tau\alpha(x)),\quad\Delta(x)=\frac{1}{2}(\alpha(x)+\tau\alpha(x)),\quad \forall x\in P.
\end{eqnarray}
By Proposition \ref{cbialgequ}, $(P,[\;,\;],\circ,\delta,\Delta)$
is a Poisson bialgebra. Let $x\in P$.  Then we have

\begin{eqnarray}\label{eq:delta}
\delta(x)=\frac{1}{2}[({\rm ad}(x)\otimes {\rm id}+{\rm id}\otimes{\rm
ad}(x))(r-\tau(r))+\underline{({\rm
id}\otimes L_\circ(x)-L_\circ(x)\otimes{\rm id})(r+\tau(r))}],
\end{eqnarray}
\begin{eqnarray}\label{eq:Delta}
\Delta(x)=\frac{1}{2}[({\rm id}\otimes
L_\circ(x)-L_\circ(x)\otimes{\rm id})(r-\tau(r))+\underline{({\rm ad}(x)\otimes {\rm
id}+{\rm id}\otimes{\rm ad}(x))(r+\tau(r))}].
\end{eqnarray}
Then $(P,[\;,\;],\alpha,r)$ is a coboundary Poisson bialgebra if
and only if there exists $r_1\in P\otimes P$ satisfying
\begin{equation}\label{eq:ss1}({\rm
id}\otimes L_\circ(x)-L_\circ(x)\otimes{\rm id})(r+\tau(r))=({\rm ad}(x)\otimes {\rm
id}+{\rm id}\otimes{\rm ad}(x))(r_1),
\end{equation}
\begin{equation}\label{eq:ss2}
({\rm ad}(x)\otimes {\rm
id}+{\rm id}\otimes{\rm ad}(x))(r+\tau(r))=({\rm id}\otimes
L_\circ(x)-L_\circ(x)\otimes{\rm id}) (r_1).
\end{equation}
Moreover, Eqs.~(\ref{eq:ss1}) and (\ref{eq:ss2}) hold if and only if
the following equations hold:
\begin{equation}\label{eq:sss1}
({\rm id}\otimes L(x)-R(x)\otimes {\rm id})(r+\tau(r)-r_1)=0,
\end{equation}
\begin{equation}\label{eq:sss2}
(L(x)\otimes {\rm id}-{\rm id}\otimes R(x))(r+\tau(r) +r_1)=0.
\end{equation}
Therefore for adm-Poisson bialgebras with $\alpha$ defined by Eq.~(\ref{eq:cobo}),
the corresponding Poisson bialgebras given in
Proposition~\ref{cbialgequ} include but not limited to coboundary
ones. Furthermore, 
if Eq.~(\ref{eq:con1})
holds (in particular, when $r$ is symmetric), then
Eqs.~(\ref{eq:sss1}) and (\ref{eq:sss2}) hold with
$r_1=r+\tau(r)$. Therefore in this case, an adm-Poisson
bialgebra with $\alpha$ defined by Eq.~(\ref{eq:cobo}) induces a
coboundary Poisson bialgebra.

Summarizing the above study, we have the following conclusion.

\begin{pro}\label{pro:corre} Let $(P,[\;,\;],\circ,\delta,\Delta)$  be a coboundary Poisson
bialgebra with  $\delta, \Delta $ defined by Eqs. $(\ref{eq:PB1})$
and $(\ref{eq:PB2})$ through $r\in P\otimes P$ respectively. Then
$(P,\c,\alpha)$ is an adm-Poisson bialgebra, where $(P,\c)$ is the
corresponding adm-Poisson algebra and $\alpha$ is defined by
Eq~(\ref{eq:alpha}) satisfying Eq.~(\ref{eq:cobo}). Conversely,
let $(P,\c,\alpha)$ be an adm-Poisson bialgebra with $\alpha$
defined by Eq.~(\ref{eq:cobo}) through $r\in P\otimes P$. Then
$(P,[\;,\;],\circ,r)$ is a coboundary Poisson bialgebra, where
$(P,[\;,\;],\circ)$ is the corresponding Poisson algebra and
$\Delta, \delta$ are defined by Eq.~(\ref{eq:Deltadelta}), if and
only if Eqs.~(\ref{eq:sss1}) and (\ref{eq:sss2}) hold for some
$r_1\in P\otimes P$. In particular, when $r$ satisfies
Eq.~(\ref{eq:con1}) or  $r$ is skew-symmetric,
  an adm-Poisson bialgebra $(P,\c,\alpha)$ with $\alpha$ defined by
Eq.~(\ref{eq:cobo}) through $r\in P\otimes P$ exactly corresponds
to a coboundary Poisson bialgebra.
\end{pro}

The relationship between the PYBE and the adm-PYBE is given as
follows.

\begin{pro}\label{pro:twoeq}
Let $(P, \c)$ be an adm-Poisson algebra and $(P, [\;,\;],\circ)$
be the corresponding Poisson algebra. Let $r\in P\otimes P$. If
$r$ is a solution of PYBE, then $r$ is a solution of adm-PYBE.
Conversely, if $r$ satisfies Eq.~(\ref{eq:con1}) and $r$ is a
solution of adm-PYBE, then $r$ is a solution of PYBE. In
particular, if $r$ is skew-symmetric, then  $r$ is a solution of
PYBE  if and only if $r$ is a solution of adm-PYBE.
\end{pro}

\begin{proof} As the same as in the proof of Lemma~\ref{lem:eqiv}, let $r=\frka+\frks\in P\otimes P$ with the skew-symmetric part $\frka$
and the symmetric part $\frks$.  Let $a^*,b^*,c^*\in P^*$. Then we
have {\small  \begin{eqnarray}
 &&( r_{23}\circ r_{12}-r_{13}\circ r_{23}-r_{12}\circ r_{13})(a^*,b^*,c^*)\nonumber\\
 \nonumber&&=-\langle\frka^\sharp(c^*)\circ \frka^\sharp(a^*),b^*\rangle-\langle\frka^\sharp(a^*)\circ \frka^\sharp(b^*),c^*\rangle-\langle\frka^\sharp(b^*)\circ \frka^\sharp(c^*),a^*\rangle+\langle\frks^\sharp(c^*)\circ \frks^\sharp(a^*),b^*\rangle\\
 \nonumber&&\mbox{}\hspace{0.5cm} -\langle\frks^\sharp(a^*)\circ \frks^\sharp(b^*),c^*\rangle-\langle\frks^\sharp(b^*)\circ \frks^\sharp(c^*),a^*\rangle-\langle\frka^\sharp(c^*)\circ \frks^\sharp(a^*),b^*\rangle+\langle\frks^\sharp(c^*)\circ \frka^\sharp(a^*),b^*\rangle\\
&&\mbox{}\hspace{0.5cm}-\langle\frka^\sharp(a^*)\circ
\frks^\sharp(b^*),c^*\rangle-\langle\frks^\sharp(a^*)\circ
\frka^\sharp(b^*),c^*\rangle+\langle\frka^\sharp(b^*)\circ
\frks^\sharp(c^*),a^*\rangle+\langle\frks^\sharp(b^*)\circ
\frka^\sharp(c^*),a^*\rangle,\label{eq:AYBE1}\\
\nonumber &&( [r_{23}, r_{12}]+[r_{23},r_{13}]+[r_{13},r_{12}])(a^*,b^*,c^*)\\
 \nonumber&&=-\langle[\frka^\sharp(c^*), \frka^\sharp(a^*)],b^*\rangle-\langle[\frka^\sharp(a^*), \frka^\sharp(b^*)],c^*\rangle-\langle[\frka^\sharp(b^*),\frka^\sharp(c^*)],a^*\rangle+\langle[\frks^\sharp(c^*), \frks^\sharp(a^*)],b^*\rangle\\
\nonumber &&-\langle[\frks^\sharp(a^*), \frks^\sharp(b^*)],c^*\rangle-\langle[\frks^\sharp(b^*), \frks^\sharp(c^*)],a^*\rangle-\langle[\frka^\sharp(c^*), \frks^\sharp(a^*)],b^*\rangle+\langle[\frks^\sharp(c^*), \frka^\sharp(a^*)],b^*\rangle\\
 &&-\langle[\frka^\sharp(a^*),
\frks^\sharp(b^*)],c^*\rangle-\langle[\frks^\sharp(a^*),
\frka^\sharp(b^*)],c^*\rangle+\langle[\frka^\sharp(b^*),
\frks^\sharp(c^*)],a^*\rangle+\langle[\frks^\sharp(b^*),
\frka^\sharp(c^*)],a^*\rangle.\label{eq:YBE1}
 \end{eqnarray}}
 By Eqs. \eqref{eq:adm-PYBE1} and \eqref{eq:PA-sPA}, we have
 \begin{eqnarray*}
   &&( r_{23}\c r_{12}-r_{13}\c r_{23}-r_{12}\c r_{13})(a^*,b^*,c^*)\\
   &=&( r_{23}\circ r_{12}-r_{13}\circ r_{23}-r_{12}\circ r_{13})(a^*,b^*,c^*)+( [r_{23}, r_{12}]+[r_{23},r_{13}]+[r_{13},r_{12}])(a^*,b^*,c^*).
 \end{eqnarray*}
 Thus if $r$ is a solution of PYBE, then $r$ is a solution of adm-PYBE.

 Conversely, suppose that  $r$ is a solution of adm-PYBE. By Eqs.~(\ref{eq:sPA-PA}) and \eqref{eq:AYBE1}, we have
 \begin{eqnarray*}
&&2( r_{23}\circ r_{12}-r_{13}\circ r_{23}-r_{12}\circ r_{13})(a^*,b^*,c^*)\\
&=&( r_{23}\c r_{12}-r_{13}\c r_{23}-r_{12}\c r_{13})(a^*,b^*,c^*)+( r_{23}\c r_{12}-r_{13}\c r_{23}-r_{12}\c r_{13})(b^*,a^*,c^*)\\
&&+2\big(\langle\frks^\sharp(a^*)\star \frks^\sharp(c^*),b^*\rangle-\langle\frks^\sharp(c^*)\star \frks^\sharp(b^*),a^*\rangle-\langle\frks^\sharp(a^*)\star \frka^\sharp(c^*),b^*\rangle+\langle\frka^\sharp(c^*)\star \frks^\sharp(b^*),a^*\rangle\big).
 \end{eqnarray*}
 By Eq.~(\ref{eq:con1}), we have
 \begin{eqnarray*}
   \langle\frks^\sharp(a^*)\star \frks^\sharp(c^*),b^*\rangle=\langle\frks^\sharp(c^*)\star
   \frks^\sharp(b^*),a^*\rangle,\;\;
   \langle\frks^\sharp(a^*)\star \frka^\sharp(c^*),b^*\rangle=\langle\frka^\sharp(c^*)\star \frks^\sharp(b^*),a^*\rangle.
 \end{eqnarray*}
 Then  $r$ satisfies
 $$r_{23}\circ r_{12}-r_{13}\circ r_{23}-r_{12}\circ r_{13}=0,$$
that is, $r$ is a solution of AYBE. Similarly, $r$ is a solution
of CYBE. Hence  
 $r$ is a solution of PYBE.
\end{proof}


There is an equivalent expression of Theorem~
\ref{cLie-Poissonbialg} as follows.

\begin{cor}\label{corollary1}
Let $(P,[\;,\;],\circ)$ be a Poisson algebra and $r\in P\otimes
P$. Suppose that the linear maps $\delta: P\to \wedge^2P$ and $
\Delta:P\otimes P\to P$ are defined by Eqs. $(\ref{eq:PB1})$ and
$(\ref{eq:PB2})$ respectively. Then
$(P,[\;,\;],\circ,\delta,\Delta)$ is a Poisson bialgebra if and
only if the following conditions are satisfied:
\begin{itemize}
    \item[\textup{${\rm(a)}$}]
    $\big((L_\circ(x)+{\rm ad}(x))\otimes{\rm id}-{\rm id}\otimes (L_\circ(x)-{\rm ad}(x))\big)(r+\tau(r))=0$,
    \item[\textup{${\rm(b)}$}]
    $\big((L_\circ(x)-{\rm ad}(x))\otimes {\rm id}\otimes{\rm id}-{\rm id}\otimes{\rm id}\otimes (L_\circ(x)+{\rm ad}(x))\big)(\textup{\textbf{A}}(r)+\textup{\textbf{C}}(r))\\
    +\frac{1}{3}\Big(\big({\rm id}\otimes (L_\circ(x)-{\rm ad}(x))\otimes {\rm id}-{\rm id}\otimes{\rm id}\otimes (L_\circ(x)-{\rm ad}(x))\big)(\textup{\textbf{A}}(r)+\textup{\textbf{C}}(r))\\
    +\big((L_\circ(x)-{\rm ad}(x))\otimes{\rm id}\otimes{\rm id}-{\rm id}\otimes (L_\circ(x)-{\rm ad}(x))\otimes{\rm
    id}\big)(\textup{\textbf{A}}(r)-\textup{\textbf{C}}(r))\Big)=0$,
\end{itemize}
for all $x\in P$.
\end{cor}

\begin{proof} Suppose that $(P,[\;,\;],\circ,\delta,\Delta)$ is a Poisson
bialgebra. By Proposition \ref{pro:corre}, $(P,\c,\alpha)$ is an
adm-Poisson bialgebra, where $(P,\c)$ is the corresponding
adm-Poisson algebra and $\alpha$ is defined by
Eq~(\ref{eq:alpha}). Furthermore, by Conditions (1) and (2) in
Theorem \ref{cLie-Poissonbialg}, we have
  $$((L_\circ(x)+{\rm ad}(x))\otimes{\rm id}-{\rm id}\otimes (L_\circ(x)-{\rm ad}(x)))(r+\tau(r))=0,$$
that is, Condition (a) holds.  By Corollary \ref{pro:coboundary
adm-Poisson}, Eq.~(\ref{eq:cosp2}) holds, that is, for all $x\in
P$,
  \begin{eqnarray*}
&&(R(x)\otimes {\rm id}\otimes{\rm id}-{\rm id}\otimes{\rm
id}\otimes L(x)){\bf P}(r)+\frac{1}{3}\big(({\rm id}\otimes R(x)\otimes {\rm
id}-{\rm id}\otimes{\rm id}\otimes
R(x)){\bf P}(r)\nonumber\\
  &&+\left(R(x)\otimes{\rm id}\otimes{\rm id}-{\rm id}\otimes R(x)\otimes{\rm id}\right){\bf Q}(r)\big)=0,
    \end{eqnarray*}
where $L=L_\circ+\ad$ and $R=L_\circ-\ad$. Note that
$${\bf P}(r)=\textup{\textbf{A}}(r)+\textup{\textbf{C}}(r),\quad {\bf Q}(r)=\textup{\textbf{A}}(r)-\textup{\textbf{C}}(r).$$
Thus Condition (b) holds.

Conversely, let $(P,\star)$ be the corresponding adm-Poisson
algebra and $\alpha$ be the linear map defined by
Eq.~(\ref{eq:alpha}). If Condition (b) holds, then
Eq.~(\ref{eq:cosp2}) holds with $L=L_\circ+\ad$ and
$R=L_\circ-\ad$. Therefore by Condition (a) and Corollary
\ref{pro:coboundary adm-Poisson}, $(P,\c,\alpha)$ is an
adm-Poisson bialgebra. By Proposition \ref{pro:corre} and
Condition (a) again, $(P,[\;,\;],\circ,\delta,\Delta)$ is a
Poisson bialgebra.
\end{proof}

\delete{
\begin{proof}
If $(P,[\;,\;],\circ,\delta,\Delta)$ is a Poisson bialgebra, recall from Theorem \ref{cLie-Poissonbialg} that Eqs.~(1)-(5) hold, which is equivalent to $A(x)=0,B(x)=0,C(x)=0,D(x)=0,E(x)=0,\ \forall x\in P$, where
\begin{align*}
A(x)=&({\rm ad}(x)\otimes {\rm id}+{\rm id}\otimes{\rm ad}(x))(r+\tau(r));\\
B(x)=&(L_{\circ}(x)\otimes {\rm id}-{\rm id}\otimes L_{\circ}(x))(r+\tau(r));\\
C(x)=&({\rm ad}(x)\otimes {\rm id}\otimes {\rm id}+{\rm id}\otimes{\rm ad}(x)\otimes {\rm id}+{\rm id}\otimes {\rm id}\otimes {\rm ad}(x))\textup{\textbf{C}}(r);\\
D(x)=&(L_{\circ}(x)\otimes {\rm id}\otimes {\rm id}-{\rm id}\otimes{\rm id}\otimes L_{\circ}(x))\textup{\textbf{A}}(r);\\
E(x)=&-({\rm ad}(x)\otimes {\rm id}\otimes {\rm id})\textup{\textbf{A}}(r)+({\rm id}\otimes L_{\circ}(x)\otimes{\rm id}
    -{\rm id}\otimes{\rm id}\otimes L_{\circ}(x))\textup{\textbf{C}}(r).
\end{align*}
Now define
\begin{align*}
H(x)=&[(L_\circ(x)-{\rm ad}(x))\otimes{\rm id}-{\rm id}\otimes (L_\circ(x)+{\rm ad}(x))](r+\tau(r));\\
G(x)=&\big[(L_\circ(x)-{\rm ad}(x))\otimes {\rm id}\otimes{\rm id}-{\rm id}\otimes{\rm id}\otimes (L_\circ(x)+{\rm ad}(x))\big](-\textup{\textbf{C}}(r)-\textup{\textbf{A}}(r))\\
    &+\frac{1}{3}\big[{\rm id}\otimes (L_\circ(x)-{\rm ad}(x))\otimes {\rm id}-{\rm id}\otimes{\rm id}\otimes (L_\circ(x)-{\rm ad}(x))\big](-\textup{\textbf{C}}(r)-\textup{\textbf{A}}(r))\\
    &+\frac{1}{3}\big[(L_\circ(x)-{\rm ad}(x))\otimes{\rm id}\otimes{\rm id}-{\rm id}\otimes (L_\circ(x)-{\rm ad}(x))\otimes{\rm id}\big](\textup{\textbf{C}}(r)-\textup{\textbf{A}}(r)).
\end{align*}
By a straightforward calculation, we have
\begin{align*}
H(x)=&B(x)-A(x),\\
G(x)=&\frac{2}{3}C(x)-\frac{4}{3}D(x)-\frac{4}{3}E(x)-\frac{2}{3}\tau_{321}E(x)\\
&-\frac{2}{3}({\rm id}\otimes{\rm id}\otimes{\rm ad}(x))[-\tau_{12}(b_i\otimes B(a_i)+a_i\otimes B(b_i)]\\
&+\frac{2}{3}(L_\circ(x)\otimes{\rm id}\otimes{\rm id}-{\rm id}\otimes L_\circ(x)\otimes{\rm id})[\tau_{12}(b_i\otimes(A(a_i))+a_i\otimes A(b_i)],
\end{align*}
so $H(x)=0,\ G(x)=0$, therefore we have Eqs.~(a)-(b).

On the other hand, in the case
$$\delta(x)=({\rm ad}(x)\otimes {\rm id}+{\rm id}\otimes{\rm ad}(x))r,\quad \Delta(x)=({\rm id}\otimes L_{\circ}(x)-L_{\circ}(x)\otimes {\rm id})r,$$
the corresponding adm-Poisson algebra is equipped with
$$\alpha(x)=\delta(x)+\Delta(x)=({\rm id}\otimes L(x)-R(x)\otimes{\rm id})(r).$$

Notice that
\begin{align*}
M(r)&=r_{23}\c r_{12}-r_{13}\c r_{23}-r_{12}\c r_{13}=-\textup{\textbf{C}}(r)-\textup{\textbf{A}}(r),\\
N(r)&=r_{12}\c r_{23}-r_{23}\c r_{13}-r_{13}\c r_{12}=\textup{\textbf{C}}(r)-\textup{\textbf{A}}(r),
\end{align*}
hence we could see that Eqs.~(\ref{eq:eqv1})-(\ref{eq:eqv3}) hold
due to the condition $(a)$, and the combination of $(a)$ and $(b)$
guarantees that Eqs.~(\ref{eq:cosp}) holds. According to Theorem
\ref{co2}, $(P,\c,\alpha)$ is an adm-Poisson bialgebra.

Now applying Proposition \ref{cbialgequ}, $(P,[\;,\;],\circ,\tilde{\delta},\tilde\Delta)$ is a Poisson bialgebra, with
$$\tilde\delta=\frac{1}{2}\left(\alpha-\tau\alpha\right),\quad \tilde\Delta=\frac{1}{2}\left(\alpha+\tau\alpha\right).$$

Notice that $(a)+\tau(a)$ gives $({\rm ad}(x)\otimes {\rm id}+{\rm id}\otimes{\rm ad}(x))(r+\tau(r))=0$, $(a)-\tau(a)$ gives $(L_{\circ}(x)\otimes {\rm id}-{\rm id}\otimes L_{\circ}(x))(r+\tau(r))$, hence
\begin{align*}
\tilde\delta(x)&=\frac{1}{2}\left(\alpha(x)-\tau\alpha(x)\right)\\
&=\frac{1}{2}[({\rm ad}(x)\otimes {\rm id}+{\rm id}\otimes{\rm
ad}(x))(r-\tau(r))+({\rm
id}\otimes L_\circ(x)-L_\circ(x)\otimes{\rm id})(r+\tau(r))]\\
&=({\rm ad}(x)\otimes {\rm id}+{\rm id}\otimes{\rm ad}(x))r=\delta(x);\\
\tilde\Delta(x)&=\frac{1}{2}\left(\alpha(x)+\tau\alpha(x)\right)\\
&=\frac{1}{2}[({\rm id}\otimes
L_\circ(x)-L_\circ(x)\otimes{\rm id})(r-\tau(r))+({\rm ad}(x)\otimes {\rm
id}+{\rm id}\otimes{\rm ad}(x))(r+\tau(r))]\\
&=({\rm id}\otimes L_{\circ}(x)-L_{\circ}(x)\otimes {\rm id})r=\Delta(x),
\end{align*}
therefore $(P,[\;,\;],\circ,\delta,\Delta)$ is a coboundary Poisson bialgebra as desired.

\cm{why? some details should be given}

\cm{whether the above is enough? the following is needed?}

\cm{The whole proof looks a little confused. We need clarify every
detail.}

\jt{The detailed proof is provided. Please check whether this new version makes sense.}
\end{proof}

\cm{Maybe we should rewrite the whole conclusion and proof.}

\cm{It is not obvious? Why?}

\cm{In a summary, the last part involving Corollary 5.15 should be
rewritten.}}

\section{$\mathcal{O}$-operators of adm-Poisson algebras and pre-adm-Poisson algebras\label{Opre}}

In this section, we introduce the notions of
$\mathcal{O}$-operators of adm-Poisson algebras and
pre-adm-Poisson algebras to construct skew-symmetric solutions of
adm-Poisson Yang-Baxter  equation and hence to construct the
induced adm-Poisson bialgebras. Note the notion of pre-adm-Poisson
algebras given here is an equivalent presentation for the notion
of pre-Poisson algebras given by Aguiar in \cite{Aguiar2}, like
the correspondence between the presentation with one operation and
the usual presentation for Poisson algebras.

\begin{defi}
Let $(P,\c)$ be an adm-Poisson algebra and $(\frkl,\frkr,V)$ be a
representation of $(P,\c)$. A linear map $\theta:V\to P$ is called
an \textbf{\textup{$\mathcal{O}$-operator of $(P,\c)$ associated
to}} $(\frkl,\frkr,V)$ if $\theta$ satisfies
    \begin{equation}
    \theta(u)\c \theta(v)=\theta(\frkl(\theta(u))v+\frkr(\theta(v))u),\quad \forall u,v\in V.
    \end{equation}
\end{defi}


\begin{ex}
Let $(P,\c)$ be an adm-Poisson algebra. An $\mathcal O$-operator
$R$ associated to the representation $(L,R,P)$ is called a {\bf
Rota-Baxter operator of weight zero}, that is, $R$ satisfies
\begin{equation}
R(x)\star R(y)=R(R(x)\star y+x\star R(y)),\;\;\forall x,y\in P.
\end{equation}
\end{ex}

\begin{ex}
Let $(P,\c)$ be an adm-Poisson algebra and $r\in P\otimes P$. If
$r$ is skew-symmetric, then by Proposition~\ref{pro:of}, $r$ is a
solution of adm-Poisson Yang-Baxter equation in $(P,\star)$ if and
only if $r^\sharp:P^*\rightarrow P$ is an $\mathcal O$-operator
associated to the representation $(-R^*,-L^*,P^*)$.
\end{ex}

There is the following construction of (skew-symmetric) solutions
of adm-Poisson Yang-Baxter equation in a semi-direct product
adm-Poisson algebra from an $\mathcal O$-operator of an
adm-Poisson algebra which is similar as for associative algebras
(\cite[Theorem 2.5.5]{Bai2}, hence the proof is omitted).

\begin{thm}\label{theorem2}
Let $(P,\c)$ be an adm-Poisson algebra and $(\frkl,\frkr,V)$ be a
representation of the adm-Poisson algebra $(P,\c)$. Let
$\theta:V\to P$ be a linear map which is identified as an element
in $P\ltimes_{-\frkr^*,-\frkl^*} V^*\otimes P\ltimes_{-\frkr^*,-\frkl^*} V^*$.
Then $r=\theta-\tau(\theta)$ is a skew-symmetric solution of
adm-PYBE in $P\ltimes_{-\frkr^*,-\frkl^*} V^*$ if and only if $\theta$ is
an $\mathcal{O}$-operator of $(P,\c)$ associated to the
representation $(\frkl,\frkr,V)$.
\end{thm}

\begin{defi}
\textbf{\textup{A pre-adm-Poisson algebra}} is a triple
$(A,\succ,\prec)$ such that $A$ is a vector space,
$\succ,\prec:A\otimes A\to A$ are two bilinear operations
satisfying the following conditions:
\begin{eqnarray}
A(x,y,z):&=&-(x\succ y)\succ z-(x\prec y)\succ z+x\succ(y\succ z)+\frac{1}{3}\big(x\succ(z\prec y)\nonumber\\
      &&-z\prec(x\succ y)-z\prec(x\prec y)-y\succ(x\succ z)+y\succ(z\prec x)\big)=0,\label{eq:presp1}\\
B(x,y,z):&=&-x\succ(z\prec y)+(x\succ z)\prec y+\frac{1}{3}\big(-x\succ(y\succ z)+y\succ(x\succ z)\nonumber\\
      &&+z\prec(x\prec y)+z\prec(x\succ y)-z\prec(y\succ x)-z\prec(y\prec x)\big)=0,\label{eq:presp2}\\
C(x,y,z):&=&-z\prec(x\succ y)-z\prec(x\prec y)+(z\prec x)\prec y+\frac{1}{3}\big(-z\prec(y\succ x)\nonumber\\
     &&-z\prec(y\prec x)+y\succ(z\prec x)+x\succ(z\prec y)-x\succ(y\succ z)\big)=0\label{eq:presp3}
\end{eqnarray}
for all $x,y,z\in A$.
\end{defi}

\begin{pro}
Let $(A,\succ,\prec)$ be a pre-adm-Poisson algebra. Define
  \begin{equation}\label{eq:sum}x\c y=x\succ y+x\prec y,\quad \forall x,y\in
  A.\end{equation}
Then $(A,\star)$ is an adm-Poisson algebra, which is called the
\textup{\textbf{sub-adjacent adm-Poisson algebra}} of
$(A,\succ,\prec)$ and denoted by $A^c$ and $(A,\succ,\prec)$ is
called  the {\bf compatible pre-adm-Poisson algebra} structure on
the adm-Poisson algebra $A^c$.
\end{pro}\label{proposition1}

\begin{proof}
Let $(A,\succ,\prec)$ be a pre-adm-Poisson algebra. For all
$x,y,z\in A$, we have
\begin{align*}
&-(x\c y)\c z+x\c(y\c z)+\frac{1}{3}\big(-x\c(z\c y)+z\c(x\c y)+y\c(x\c z)-y\c(z\c x)\big)\\
=&\;A(x,y,z)+B(x,y,z)+C(x,y,z),
\end{align*}
\emptycomment{where
\begin{eqnarray*}
A(x,y,z)&=&-(x\succ y)\succ z-(x\prec y)\succ z+x\succ(y\succ z)+\frac{1}{3}\big(x\succ(z\prec y)-z\prec(x\succ y)\nonumber\\
      &&-z\prec(x\prec y)-y\succ(x\succ z)+y\succ(z\prec x)\big),\\
B(x,y,z)&=&-x\succ(z\prec y)+(x\succ z)\prec y+\frac{1}{3}\big(-x\succ(y\succ z)+y\succ(x\succ z)+z\prec(x\prec y)\nonumber\\
&&+z\prec(x\succ y)-z\prec(y\succ x)-z\prec(y\prec x)\big),\\
C(x,y,z)&=&-z\prec(x\succ y)-z\prec(x\prec y)+(z\prec x)\prec y+\frac{1}{3}\big(-z\prec(y\succ x)-z\prec(y\prec x)\nonumber\\
     &&+y\succ(z\prec x)+x\succ(z\prec y)-x\succ(y\succ z)\big).
\end{eqnarray*}}
By Eqs.~(\ref{eq:presp1})-(\ref{eq:presp3}), we have
$A(x,y,z)=B(x,y,z)=C(x,y,z)=0$. Hence $(A^c,\c)$ is an adm-Poisson
algebra.
\end{proof}

\begin{rmk}
In fact, the operad $PreadmPois$ of pre-adm-Poisson algebras is
 the disuccessor (splitting an operad into two pieces) of the operad
$admPois$ of adm-Poisson algebras in the sense of \cite{BBGN}.
Note that the operad $Dend$ of dendriform algebras introduced by
Loday (\cite{Loday}) is the disuccessor of the operad $Ass$ of
associative algebras. Hence pre-adm-Poisson algebras can be
regarded as analogue structures of dendriform algebras with many
similar properties (\cite{Bai2,BGN2}).
\end{rmk}

\begin{pro}
Let $(A,\succ,\prec)$ be a pre-adm-Poisson algebra. Then
$(L_\succ,R_\prec,A)$ is a representation of the sub-adjacent
 adm-Poisson algebra $(A^c,\c)$, where $L_\succ,R_\prec:A\rightarrow \End_{\mathbb F}(A)$ are defined by
  \begin{equation}
   L_\succ(x)y=x\succ y,\quad R_\prec(x)y=y\prec x,\quad\forall x,y\in A.
  \end{equation}
  Conversely, if $(A,\c)$ is an adm-Poisson algebra together with two bilinear operations $\succ,\prec:A\otimes A\to A$
such that $(L_\succ,R_\prec,A)$ is a representation of $(A,\c)$,
then $(A,\succ,\prec)$ is a pre-adm-Poisson algebra.
\end{pro}

\begin{proof}
Eq. \eqref{eq:presp1} implies that Eq. \eqref{c2} holds, Eq.
\eqref{eq:presp2} implies that Eq. \eqref{c3} holds and Eq.
\eqref{eq:presp3} implies that Eq. \eqref{c4} holds with
$\frkl=L_\succ$ and $\frkr=R_\prec$. Thus $(L_\succ,R_\prec,A)$ is a
representation of the sub-adjacent adm-Poisson algebra $(A^c,\c)$.
The converse can be proved similarly. \emptycomment{In fact, for
all $x,y,z \in A$, we have
\begin{eqnarray*}
 &&-L_\succ(a\star b)+L_\succ(a)L_\succ(b)-\frac{1}{3}\big(-L_\succ(a)R_\prec(b)+R_\prec(a\c b)+L_\succ(b)L_\succ(a)-L_\succ(b)R_\prec(a)\big)=0, \\
 &&-R_\prec(b)L_\succ(a)+L_\succ(a)R_\prec(b)-\frac{1}{3}\big[-L_\succ(a)L_\succ(b)+L_\succ(b)L_\succ(a)+R_\prec(a\star b)-R_\prec(b\star a)\big]=0,\\
 &&-R_\prec(b)R_\prec(a)+R_\prec(a\star b)-\frac{1}{3}\big[-R_\prec(b\star a)+L_\succ(b)R_\prec(a)+L_\succ(a)R_\prec(b)-L_\succ(a)L_\succ(b)\big]=0,
\end{eqnarray*}
hold if and only if Eqs.~(\ref{eq:presp1})-(\ref{eq:presp3}) hold
respectively, by letting the right-hand sides of Eqs.
$(a),(b),(c)$ act on an arbitrary element $c$ in $A$.}
\end{proof}

A direct consequence is given as follows.

\begin{cor} \label{cor:id}
Let $(A, \prec, \succ)$ be a pre-adm-Poisson algebra. Then the
identity map ${\rm id}$ is an $\mathcal O$-operator of the
sub-adjacent adm-Poisson algebra $(A^c,\c)$ associated to the
representation $(L_\succ,R_\prec$, $A)$.
\end{cor}

\begin{defi}
\begin{enumerate}
\item {\rm (\cite{G1})} A left \textbf{pre-Lie algebra} is a
vector space $A$ together with a bilinear operation $*:A\otimes
A\to A$ such that
\begin{equation}
    x*(y*z)-(x*y)*z=y*(x*z)-(y*x)* z,\;\;\forall x,y,z\in
    A.\end{equation}
\item {\rm (\cite{Lod2})} A left \textbf{Zinbiel algebra} is a
vector space $A$ together with a bilinear operation
$\cdot:A\otimes A\to A$ such that
\begin{equation}
    x\cdot(y\cdot z)=(y\cdot x)\cdot z+(x\cdot y)\cdot
    z,\;\;\forall x,y,z\in A.
\end{equation}

\item {\rm (\cite{Aguiar2})} A left \textbf{ pre-Poisson
algebra} is a triple $(A,\cdot,*)$ such that $(A,\cdot)$ is a
left Zinbiel algebra, $(A,*)$ is a left pre-Lie algebra and the
following conditions hold:
\begin{eqnarray}
&&(x*y-y*x)\cdot z=x*(y\cdot z)-y\cdot(x*z),\\
&&(x\cdot y+y\cdot x)*z=x\cdot(y*z)+y\cdot(x*z),\;\forall x,y,z\in
A.
\end{eqnarray}
\end{enumerate}
\end{defi}

\begin{pro} {\rm (\cite{Aguiar2})}
Let $(A,\cdot,*)$ be a left pre-Poisson algebra. Define
\begin{align*}
    x\circ y=x\cdot y+y\cdot x, \quad [x,y]=x*y-y*x,\quad \forall x,y\in A.
\end{align*}
Then $(A,[\;,\;],\circ)$ is a Poisson algebra.
\end{pro}

There exists a one-to-one correspondence between pre-adm-Poisson
algebras and pre-Poisson algebras.
\begin{pro}
If $(A,\succ,\prec)$ is a pre-adm-Poisson algebra, define
\begin{equation*}
x\cdot y=\frac{1}{2}(x\succ y+y\prec x),\quad x*y=\frac{1}{2}(x\succ y-y\prec x),\quad\forall x,y\in A,
\end{equation*}
then $(A,\cdot,*)$ is a pre-Poisson algebra. Conversely, if $(A,\cdot,*)$ is a pre-Poisson algebra, define
\begin{equation*}
x\succ y=x\cdot y+x*y,\quad x\prec y=y\cdot x-y*x,\quad \forall x,y\in A,
\end{equation*}
then $(A,\succ,\prec)$ is a pre-adm-Poisson algebra.
\end{pro}

\begin{proof}
It is straightforward.
\end{proof}

\begin{thm}\label{thm:pre-adm} Let $(P,\c)$ be an adm-Poisson algebra and $\theta:V\to P$ be an
$\mathcal{O}$-operator of the adm-Poisson algebra $P$ associated
to the representation $(\frkl,\frkr,V)$. Then there exists a
pre-adm-Poisson algebra structure on $V$ given by
\begin{equation}\label{eq:Vp}
u\succ v=\frkl(\theta(u))v,\;\;u\prec v=\frkr(\theta(v))u,\;\;\forall\;
u,v\in V.\end{equation} So there is the sub-adjacent adm-Poisson
algebra structure on $V$ given by Eq.~(\ref{eq:sum}) and $\theta$
is a homomorphism of adm-Poisson algebras. Moreover,
$\theta(V)=\{\theta(v)|v\in V\}\subset P$ is an adm-Poisson
subalgebra of $(P,\c)$ and there is an induced pre-adm-Poisson
algebra structure on $\theta(V)$ given by
\begin{equation}\label{eq:Ap}\theta(u)\succ \theta(v)=\theta(u\succ v),\;\;\theta(u)\prec \theta(v)=\theta(u\prec v),\;\;\forall\; u,v\in V.\end{equation}
Its corresponding sub-adjacent adm-Poisson algebra structure on
$T(V)$ given by Eq.~(\ref{eq:sum}) is just the adm-Poisson
subalgebra structure of $(P,\c)$ and $\theta$ is a homomorphism of
pre-adm-Poisson algebras.\end{thm}

\begin{proof}
By Eq. $(\ref{c2})$, for all $u,v,w\in V$, we have
\begin{eqnarray*}
     &&-(u\succ v)\succ w-(u\prec v)\succ w+u\succ(v\succ w)\\
     &&+\frac{1}{3}\big(u\succ(w\prec v)-w\prec(u\succ v)-w\prec(u\prec v)-v\succ(u\succ w)+v\succ(w\prec u)\big)\\
&=&-\frkl(\theta(\frkl(\theta (u))v))w-\frkl(\theta(\frkr(\theta v)u))w+\frkl(\theta (u))\frkl(\theta (v))w+\frac{1}{3}\big(\frkl(\theta (u))\frkr(\theta (v))w\\
 &&-\frkr(\theta(\frkl(\theta(u))v))w-\frkr(\theta(\frkr(\theta (v))u))w-\frkl(\theta (v))\frkl(\theta(u))w+\frkl(\theta(v))\frkr(\theta(u))w\big)\\
&=&-\frkl(\theta (u)\c\theta (v))+\frkl(\theta (u))\frkl(\theta(v))+\frac{1}{3}\big((\frkl(\theta(v))\frkr(\theta (u))-\frkr(\theta(u)\c\theta(v))\\
 &&-\frkl(\theta(v))\frkl(\theta (u))+\frkl(\theta(v))\frkr(\theta(u)))\big)w=0.
\end{eqnarray*}
Hence Eq.~\eqref{eq:presp1} holds. Similarly, by Eq. $(\ref{c3})$,
Eq.~\eqref{eq:presp2} holds and by Eq. $(\ref{c4})$,
Eq.~\eqref{eq:presp3} holds. Therefore, $(V,\prec, \succ)$ is a
pre-adm-Poisson algebra. The rest is straightforward.
\end{proof}

\delete{

We give the following two conclusions as the direct consequences
of Theorem~\ref{thm:pre-adm} by omitting proofs since they are
similar as in the case of associative algebras (\cite[Corollary
3.1.3 and Theorem 4.1.1]{Bai2}).

\cm{Please consider whether we use the former expression with
proofs or the present expression without proofs.}

\liu{In my opinion, the present expression without proofs is suitable because these two corollaries are not the important conclusions in this section.}
\begin{cor}\label{cor:exist}
Let $(P,\c)$ be an adm-Poisson algebra. Then there exists a
compatible pre-adm-Poisson algebra structure on $P$ if and only if
there exists an invertible $\mathcal{O}$-operator of $(P,\c)$.
\end{cor}

\delete{
\begin{pf}
Let $\theta:V\to P$ be an invertible $\mathcal{O}$-operator on the
adm-Poisson algebra $(P,\c_P)$ associated to representation
$(l,r,V)$. Then the operations $\succ$ and $\prec$ given by
\begin{equation*}
x\succ y=\theta(l(x)(\theta^{-1}(y))),\quad x\prec
y=\theta(r(y)(\theta^{-1}(x))),\quad \forall\ x,y\in P
\end{equation*}
defines a compatible pre-adm-Poison algebra.

Conversely, let $(P,\succ,\prec)$ be a pre-adm-Poisson algebra
such that $(P,\c)$ is the sub-adjacent adm-Poisson algebra. Then
by Corollary~\ref{cor:id},  ${\rm id}$ is an
$\mathcal{O}$-operator of $(P,\c)$ associated to the
representation $(L_\succ,R_\prec,P)$.
\end{pf}}

\begin{cor}  Let $(P,\c)$ be an adm-Poisson algebra and $\omega$ be a
nondegenerate skew-symmetric bilinear form satisfying Eq.~
(\ref{eq:Conn}). Then there exists a pre-adm-Poisson algebra
structure $\succ,\prec$ on $P$ given by
\begin{equation}\label{eq:Conn-pre}\omega(x\succ y,z)=\omega(y, z\c x),\;\; \omega(x\prec
y,z)=\omega(x,y\c z),\;\; \forall x,y,z\in A\end{equation} such
that the sub-adjacent adm-Poisson algebra is $(P,\c)$.
\end{cor}

\delete{

\begin{proof} Define a linear map $\theta:A\rightarrow A^*$ by
$$\langle \theta(x),y\rangle=\omega (x,y),\;\;\forall x,y\in A.$$ Then  $\theta^{-1}$ is an $\mathcal O$-operator of  $(P,\c)$ associated to the bimodule
$(-R^*,-L^*, P^*)$. By Corollary~ \ref{cor:exist}, there is a
pre-adm-Poisson
 algebra structure $\succ,\prec$ on $(P,\c)$ given by
$$x\succ y=-\theta^{-1}R^*(x)\theta(y),\;\;x\prec y=-\theta^{-1}L^*(y)\theta(x),\;\;\forall
x,y\in P,$$ which gives exactly Eq.~(\ref{eq:Conn-pre}) such that
the sub-adjacent adm-Poisson algebra is $(P,\c)$.
\end{proof}}}

Finally we give the following construction of skew-symmetric
solutions of adm-Poisson Yang-Baxter equation (hence adm-Poisson
bialgebras) from pre-adm-Poisson algebras.

\begin{pro}
 Let $(A,\succ,\prec)$ be a pre-adm-Poisson algebra.
Then
\begin{equation}
r=\sum_{i=1}^n (e_i\otimes e_i^*-e_i^*\otimes e_i)\end{equation}
is a solution of adm-Poisson Yang-Baxter equation in the
adm-Poisson algebra $A^c \ltimes_{-R_\prec^*,-L_\succ^*} A^*$,
where $\{e_1,\cdots, e_n\}$ is a basis of $A$ and $\{e_1^*,\cdots,
e_n^*\}$ is its dual basis.
\end{pro}

\begin{proof} 
Note that ${\rm id}=\sum\limits_{i=1}^n e_i\otimes e_i^*$. The
conclusion follows from Theorem~\ref{theorem2} and
Corollary~\ref{cor:id}.
 \end{proof}

\bigskip

 \noindent
 {\bf Acknowledgements.}  This work is supported by
NSFC (11931009, 11901501).  C. Bai is also supported by the
Fundamental Research Funds for the Central Universities and Nankai
Zhide Foundation.

\end{document}